\newif \ifJournal \Journalfalse

\ifJournal

\documentclass{ws-m3as}
\usepackage[super]{cite}
\usepackage{url}

\else

\documentclass[11pt,a4paper]{article}
\usepackage[utf8]{inputenc}
\usepackage[T1]{fontenc}
\usepackage[english]{babel}
\fi

\usepackage{amsmath,amsfonts,amssymb,mathtools}
\usepackage{mathrsfs}
\usepackage{graphicx}
\usepackage{caption} 
\usepackage{wrapfig}
\usepackage[dvipsnames, svgnames]{xcolor}
\usepackage{dsfont}
\usepackage{empheq}
\usepackage{thmtools}
\usepackage{thm-restate}
\usepackage[normalem]{ulem}
\usepackage{upgreek}
\usepackage{nicefrac}
\usepackage{cancel}
\usepackage{comment}
\usepackage{enumerate}
\usepackage{tikzscale}
\usepackage{paralist}
\usepackage{pgfplots}
\pgfplotsset{every axis/.append style={
                    label style={font=\Large},
                    tick label style={font=\Large},
                    legend style={font=\Large}
                    }}
\usepackage{subcaption}
\usepackage{etoolbox}
\usepackage[skins]{tcolorbox}
\makeatletter
\patchcmd{\ttlh@hang}{\parindent\z@}{\parindent\z@\leavevmode}{}{}
\patchcmd{\ttlh@hang}{\noindent}{}{}{}
\makeatother
\usepackage{multirow}
\usetikzlibrary{patterns}
\usepackage[scale = .8]{geometry}

\usepackage{tabularx}
\newcolumntype{Y}{>{\centering\arraybackslash}X}

\pgfplotsset{compat=1.15}
\newcommand{\logLogSlopeTriangle}[5]
{
    \pgfplotsextra
    {
        \pgfkeysgetvalue{/pgfplots/xmin}{\xmin}
        \pgfkeysgetvalue{/pgfplots/xmax}{\xmax}
        \pgfkeysgetvalue{/pgfplots/ymin}{\ymin}
        \pgfkeysgetvalue{/pgfplots/ymax}{\ymax}

        \pgfmathsetmacro{\xArel}{#1}
        \pgfmathsetmacro{\yArel}{#3}
        \pgfmathsetmacro{\xBrel}{#1-#2}
        \pgfmathsetmacro{\yBrel}{\yArel}
        \pgfmathsetmacro{\xCrel}{\xArel}

        \pgfmathsetmacro{\lnxB}{\xmin*(1-(#1-#2))+\xmax*(#1-#2)} 
        \pgfmathsetmacro{\lnxA}{\xmin*(1-#1)+\xmax*#1} 
        \pgfmathsetmacro{\lnyA}{\ymin*(1-#3)+\ymax*#3} 
        \pgfmathsetmacro{\lnyC}{\lnyA+#4*(\lnxA-\lnxB)}
        \pgfmathsetmacro{\yCrel}{\lnyC-\ymin)/(\ymax-\ymin)}

        \coordinate (A) at (rel axis cs:\xArel,\yArel);
        \coordinate (B) at (rel axis cs:\xBrel,\yBrel);
        \coordinate (C) at (rel axis cs:\xCrel,\yCrel);

        \draw[#5]   (A)-- node[pos=0.5,anchor=north] {\scriptsize{1}}
                    (B)-- 
                    (C)-- node[pos=0.,anchor=west] {\scriptsize{#4}} 
                    (A);
    }
}

\usepackage{hyperref}
\hypersetup{
  unicode = true,                           
  pdftoolbar = true,                        
  pdfmenubar = true,                        
  pdffitwindow = true,                      
  pdfauthor = {S. Lemaire and J. Moatti},   
  pdftitle = {Structure preservation in high-order hybrid discretisations of potential-driven advection-diffusion: linear and nonlinear approaches},           
  pdfnewwindow = true,                      
  colorlinks = true,                        
  linkcolor = OliveGreen,                   
  citecolor = violet,                       
  filecolor = blue,                         
  urlcolor = magenta                        
}

\ifJournal

\else

\usepackage{amsthm}
\usepackage{authblk}

\newcommand{\email}[1]{\href{mailto:#1}{#1}}
\newtheorem{prop}{Proposition}
\newtheorem{rem}{Remark}

\newtheorem{lemma}{Lemma}
\newtheorem{definition}{Definition}

\title{Structure preservation in high-order hybrid discretisations of \rev{potential-driven advection-diffusion}: linear and nonlinear approaches}
\author[1]{Simon Lemaire\footnote{\email{simon.lemaire@inria.fr} (corresponding author)}}
\author[2,1]{Julien Moatti\footnote{\email{julien.moatti@tuwien.ac.at}}}
\affil[1]{Inria, Univ.~Lille, CNRS, UMR 8524 -- Laboratoire Paul Painlev\'e, F-59000 Lille, France}
\affil[2]{Institute of Analysis and Scientific Computing, Vienna University of Technology, Wiedner Hauptstr.~8--10, A-1040 Wien, Austria}

\fi

\DeclareMathOperator{\e}{e}
\DeclareMathOperator{\divergence}{div}

\newcommand{\R}{\mathbb{R}}

\newcommand{\M}{\mathcal{M}}
\newcommand{\E}{\mathcal{E}}
\newcommand{\D}{\mathcal{D}}
\newcommand{\Diss}{\mathbb{D}}
\newcommand{\Entro}{\mathbb{E}}

\newcommand{\gradd}{\mathbb{G}}

\newcommand{\poly}{\mathbb{P}}

\newcommand{\G}{\underline{\mathcal{G}}_\D}
\newcommand{\s}{\sigma}

\renewcommand{\u}{\underline{\mathfrak{u}}}
\newcommand{\ue}{\underline{\mathfrak{u}}}

\newcommand{\V}{\underline{V}}
\renewcommand{\v}{\underline{v}}
\newcommand{\w}{\underline{w}}
\newcommand{\z}{\underline{z}}
\newcommand{\ro}{\underline{\rho}}

\newcommand{\phid}{\underline{\phi}}

\renewcommand{\l}{\underline{\ell}}

\newcommand{\1}{\mathds{1}}
\newcommand{\one}{\underline{1}}
\newcommand{\zero}{\underline{0}}

\newcommand{\dd}{\mathrm{d}}
\definecolor{MyGreen}{RGB}{54,165,54}

\newcommand{\rev}[1]{{\color{black} #1}}

\begin{document}

\maketitle

\ifJournal

\else

\begin{abstract}
  We are interested in the high-order approximation of anisotropic\rev{, potential-driven advection-diffusion models} on general polytopal partitions.
  We study two hybrid schemes, both built upon the Hybrid High-Order technology.
  The first one hinges on exponential fitting and is linear, whereas the second is nonlinear.
  The existence of solutions is established for both schemes. Both schemes are also shown to possess a discrete entropy structure, ensuring that the long-time behaviour of discrete solutions mimics the PDE one.
  For the nonlinear scheme, the positivity of discrete solutions is a built-in feature.
  On the contrary, we display numerical evidence indicating that the linear scheme violates positivity, whatever the order.
  Finally, we verify numerically that the nonlinear scheme has optimal order of convergence, expected long-time behaviour, and that raising the polynomial degree results, also in the nonlinear case, in an efficiency gain.
\end{abstract}


\fi

\section{Motivations and context}

We are interested in the polytopal discretisation of a linear \rev{potential-driven} advection-diffusion equation using high-order schemes.
Our goal is to compare an exponentially fitted linear method with a nonlinear approach.
Let $\Omega$ be an open, bounded, connected polytopal subset of $\R^d$, $d \in \{2,3\}$, with Lipschitz boundary.
We consider the following anisotropic advection-diffusion problem with homogeneous Neumann boundary conditions:  
find the density $u : \R_+ \times \Omega \to \R$ solution to
\begin{equation} \label{C4:pb:evol}
	\left\{
	\begin{split}
		\partial _ t u - \divergence ( \Lambda  (\nabla u + u \nabla \phi )  ) &= 0 &&\text{ in } \R_+^\star \times \Omega, \\
		\Lambda (\nabla u + u \nabla \phi ) \cdot n &= {0} &&\text{ on } \R_+^\star \times \partial \Omega,\\		
	 	u(0,\cdot)&= u^{0} &&\text{ in } \Omega ,
	\end{split}
	\right.
\end{equation}
where $n$ is the unit normal vector to $\partial\Omega$ pointing outward from $\Omega$. We assume that the data satisfy:
\begin{itemize}
	\item[(i)]
	 $\Lambda \in L^{\infty}(\Omega)^{d\times d}$ is a uniformly elliptic anisotropy tensor: 
		there exists $\lambda_\flat > 0$ such that,
			for a.e.~$x$ in $\Omega$, $ \Lambda (x) \xi \cdot \xi \geq \lambda_\flat|\xi|^2$ for all $\xi \in \R^d$; 
	\item[(ii)]
	$\phi \in \rev{W^{1,\infty}(\Omega)}$ is a regular potential;
	\item[(iii)]
 	$u^{0} \in L^1(\Omega)$ is a non-negative initial datum, such that $\int_\Omega u^{0} \log\left(u^{0}\right) < \infty$.
\end{itemize}
The solution to \eqref{C4:pb:evol} enjoys some specific and well-known properties. 
First, the mass 
is conserved along time, i.e.~for almost every $t>0$,
\begin{equation} \label{C4:eq:pres:mass}
	\int_{\Omega}u(t)= \int_\Omega u^{0} = M,
\end{equation}
where $M >0$ is the initial mass. 
Second, the solution is positive: 
\begin{equation} \label{C4:eq:pres:posi}
	\text{ for a.e.~} t > 0, \qquad  u(t, \cdot) > 0 \, \text{ a.e.~in } \Omega. 
\end{equation} 
Last, the solution converges exponentially fast when $t \to \infty$ towards the thermal equilibrium, unique steady solution to \eqref{C4:pb:evol}, given by 
\begin{equation} \label{C4:eq:def:equi}
	u^\infty= \frac{M }{\int_\Omega \e^{-\phi}}\e^{-\phi}.
\end{equation}
In order to compute a reliable numerical approximation of Problem~\eqref{C4:pb:evol}, one should preserve at the discrete level the three above-listed structural properties.
\rev{In practice, our final target application are drift-diffusion semiconductor models~\cite{VanRo:50,GaGro:89} (and, in particular, anisotropic ones~\cite{GaGar:96}). In these models, the electric potential $\phi$ driving the drift is one of the unknowns of the problem, alongside with the densities of charge carriers. It is solution to a Poisson equation. At the PDE level, the thermal equilibrium is defined as the density $u^\infty>0$ for which the flux $\Lambda(\nabla u^\infty + u^\infty \nabla \phi)$ identically vanishes in $\Omega$. This characterisation implies that the equilibrium quasi-Fermi potential $\log(u^\infty) + \phi$ shall be constant in $\Omega$~\cite{MarUn:93}. At the discrete level, this motivates the following definition.
\begin{definition}[Preservation of the thermal equilibrium] \label{def:th.eq}
  A numerical scheme for~\eqref{C4:pb:evol} preserves the thermal equilibrium if the corresponding discrete equilibrium quasi-Fermi potential is constant.
\end{definition}
\noindent
Note that this definition implies that the discrete equilibrium density has to be positive.
In semiconductor models discretisations, the potential $\phi$ is sought as an element of the discrete space.
By Definition~\ref{def:th.eq}, preserving the thermal equilibrium then essentially requires to also seek $\log(u^\infty)$ as an element of the latter discrete space.
For the schemes we study in this work, the precise meaning of Definition~\ref{def:th.eq} will be made clear in Proposition~\ref{prop:th.eq} below.}

In the realm of Two-Point Flux Approximation (TPFA) finite volume schemes, the so-called Scharfetter--Gummel fluxes~\cite{SchGu:69} are precisely devised so as to preserve the thermal equilibrium. They naturally lead to linear structure-preserving discretisations of the problem (see~\cite{CHDro:11,CHHer:20}).
However, TPFA methods can only be used on meshes satisfying orthogonality conditions (with respect to the inner product induced by $\Lambda$, in case $\Lambda$ is symmetric), which essentially restricts their use to isotropic problems.
On the other hand, a number of finite volume methods using auxiliary unknowns has been introduced within the past twenty years or so for the discretisation of anisotropic problems on general meshes.
One can cite the Discrete Duality Finite Volume (DDFV) method~\cite{Herme:00,DomOm:05}, with additional unknowns attached to a dual mesh, the Vertex Approximate Gradient (VAG) scheme~\cite{EGHer:12}, with auxiliary unknowns attached to the mesh vertices, or the Mimetic Finite Difference (MFD) and Hybrid Finite Volume (HFV) methods~\cite{BLiSi:05,EGaHe:10}, with auxiliary unknowns attached to the mesh faces.
Such methods have proved to be relevant solutions to the anisotropy issue, but none of these linear schemes preserves the positivity of the solutions (see~\cite{Droni:14}).
A possible alternative was proposed in~\cite{CaGui:17}, with the introduction and analysis of a nonlinear positivity-preserving VAG scheme.
The design and analysis of this scheme, as well as of its DDFV and HFV counterparts of~\cite{CCHKr:18,CCHHK:20} and~\cite{CHHLM:22}, leverage the entropy structure of Problem~\eqref{C4:pb:evol}: there exists some physically motivated quantity, called entropy, which decays along time. 
Reproducing this structure at the discrete level is key to get stability, convergence, and accurate time asymptotics.
Other approaches to positivity preservation on general meshes have been explored in the literature. Still in the realm of finite volume methods, one can cite the works~\cite{SYYua:09,DroLP:11,BlaLa:16,SAEFl:17,Quenj:22}.
As opposed to~\cite{CaGui:17}, in which the nonlinearities are introduced at the PDE level then discretised, the latter contributions introduce nonlinearities directly at the discrete level. These nonlinearities, unfortunately, often do not lend themselves to a PDE re-interpretation, making difficult to unravel the potential discrete entropy structures hidden behind.
Arbitrary-order positivity-preserving (or, more generally, discrete maximum principle preserving) methods have also been studied in the literature. In the finite element context, one can cite the seminal works~\cite{Ciarl:70,CiaRa:73} by Ciarlet, as well as the more recent contributions~\cite{FaKaK:12,BGPVe:24} (see also~\cite{BJKno:24} for a comprehensive survey). These approaches are, however, restricted to standard meshes. In addition, only algebraic positivity can usually be enforced, that is positivity of the degrees of freedom, but not of the (piecewise polynomial) functions themselves over the domain. Weak positivity enforcement has also been explored in the Discontinuous Galerkin (DG) framework in~\cite{LiuYu:14,LiuWa:16}: therein, positivity is enforced on the cell averages of the piecewise polynomial solutions. Turning to pointwise positivity enforcement, let us mention in the DG context the interesting contribution~\cite{BBJPe:20}. Therein, a nonlinear scheme is introduced for the (reaction-diffusion) Fisher--KPP equation $\partial_t u - \triangle u = u(1-u)$, in which the (positive) densities are defined as $u = \e^\lambda$, with $\lambda$ piecewise polynomial. This scheme is developed so as to preserve the entropy structure of the PDE model. Compared to the high-order DG schemes of~\cite{LiuYu:14,LiuWa:16}, the main improvement lies in the fact that the discrete solutions are positive everywhere. Such a feature allows for a complete analysis of the scheme, including existence, long-time behaviour, and convergence towards a semi-discretised solution. The analysis is based on the properties of a well-chosen stabilisation term, whose expression implies $L^\infty$-norms of the polynomial unknowns over the mesh faces. The results of~\cite{BBJPe:20}, valid on simplicial meshes, have recently been extended (excluding the long-time behaviour) to polytopal meshes in~\cite{CoBoA:23}, still in the DG context. Along the same lines, yet restricted to standard meshes, let us also cite the 
conforming space-time Galerkin discretisation of~\cite{BPSto:22} for cross-diffusion systems.

From the above literature review, it is quite clear that the landscape in terms of positivity-preserving polytopal methods of arbitrarily high approximation order for advection-diffusion problems is relatively scarce. Speaking of pointwise positivity enforcement, the only existing contribution we are aware of is~\cite{CoBoA:23}, in the DG context, and for reaction-diffusion equations. In the present work, our aim is to study an arbitrary-order {\bf hybrid} polytopal scheme for Problem~\eqref{C4:pb:evol}, preserving the three structural properties~\eqref{C4:eq:pres:mass}--\eqref{C4:eq:pres:posi}--\eqref{C4:eq:def:equi} listed above.
One expected advantage of hybrid methods over DG schemes is a reduction of the number of globally coupled unknowns in the linear systems to be solved at each iteration of the Newton algorithm, which should be all the more substantial that the order of approximation increases.
Our (nonlinear) scheme has been briefly introduced, and a first numerical assessment performed, in~\cite{Moatt:23HO}.
Our goal in the present article is twofold. 
First, we want to provide our approach with theoretical foundations. Second, we aim to conduct an extensive numerical validation of our method (convergence orders, efficiency, positivity, large time), including a comparison in terms of structure preservation with a similar (in the spirit) high-order linear scheme. One could indeed expect, at least in practice, that the use of a method (even linear) with sufficiently high order (and thus accuracy), could already constitute in itself a solution to positivity violation issues.
The two (linear and nonlinear) methods we consider are built upon the Hybrid High-Order (HHO) technology~\cite{DPErn:15,DPELe:14}, as natural extensions of the HFV schemes introduced in~\cite{CHHLM:22,Moatt:23SC}.
The linear scheme hinges on the exponential fitting strategy~\cite{BMaPi:89}.
The key idea is the linear change of unknown $u=\e^{-\phi}\rho$, which allows one to reformulate~\eqref{C4:pb:evol} as an unconditionally coercive problem in the variable $\rho$. As a by-product of this reformulation, the scheme naturally preserves the thermal equilibrium.
The nonlinear scheme relies on the nonlinear change of unknown $u=\e^\ell$ (cf.~\cite{CaGui:17}), which is designed so as to preserve the Boltzmann entropy structure of the PDE model and, as a by-product, the positivity of solutions, the thermal equilibrium, and the long-time asymptotics.
For the sake of simplicity, both schemes rely on a mixed-order HHO space: given an integer $k\geq 0$, the methods hinge on face unknowns of polynomial degree $k$, and enriched cell unknowns of polynomial degree $k+1$.
The main interest of such a discrepancy in the degree between face and cell unknowns is a simplification \rev{in the design of the higher-order bulk reconstruction and of} the stabilisation, resulting in turn in a simplification of the analysis.
In the meantime, such a choice preserves optimal accuracy (order $k+2$ in $L^2$-norm) and frugality (the face unknowns, of degree $k$, are the only globally coupled unknowns).
\rev{Since we are manipulating mixed-order spaces, following~\cite{CoDPE:16}, we could also refer to our methods as HDG methods. However, we prefer naming them HHO methods for the two following reasons (cf.~\cite[Section 1.5.2]{CErPi:21}). First, HDG schemes are developed adopting a mixed-hybrid viewpoint, whereas we adopt here the primal HHO viewpoint. Second, the analysis of HDG methods usually hinges on specific (often simplex-based) projections, whereas our HHO analysis makes here a systematic use of $L^2$-orthogonal projectors (well-defined on polytopal cells). In any case, the two schemes we introduce in this work are new in the HDG/HHO context.}
Our first theoretical results, stated in Propositions~\ref{C4:prop:ExpF:exi} and~\ref{C4:prop:ExpF:time}, concern the well-posedness and long-time behaviour of the (linear) exponential fitting scheme.
Regarding the nonlinear scheme, we prove the existence of (positive) solutions in Theorem~\ref{C4:th:nl:exist}.
These discrete solutions are further proved to converge (in large time) in Proposition~\ref{C4:prop:nl:LT} towards the discrete equilibrium of the scheme.
Note that we could also have compared our nonlinear scheme with the linear HHO method for advection-diffusion of~\cite{DPDEr:15}, which generalises to arbitrary approximation orders the HMM scheme of~\cite{BdVDM:11} \rev{(both introduced for general advection fields)}. We have not pursued further in this direction, this for two reasons. First, the stability of this scheme hinges on some coercivity assumptions which constrain the variety of potentials that can be considered, and second it does not preserve the thermal equilibrium (see~\cite{CHHLM:22} in the lowest-order case).

The rest of the article is organised as follows. 
In Section~\ref{C4:sec:schemes}, we first introduce the discrete framework, and describe the two schemes under consideration. 
Then, in Section~\ref{C4:sec:prop}, we discuss the main properties of the two schemes, and we provide some elements of analysis regarding the well-posedness and discrete long-time behaviours.
Last, in Section~\ref{C4:sec:num}, we discuss the implementation of the nonlinear scheme, and we assess the behaviour of the methods on various test-cases. 

\section{Discrete setting and schemes} \label{C4:sec:schemes}

The two numerical schemes we consider in this article are based on a backward Euler discretisation in time, with uniform time step $\Delta t>0$.
The time discretisation is thus defined as $(t^n)_{n \in \mathbb{N}}$, where $t^n = n \Delta t$.  
Note that it is straightforward to generalise the discussion below to a variable time step.
We focus in this section on space discretisation.

\subsection{Polytopal meshes and anisotropy tensor}

In the vein of~\cite[Definition 1.4]{DPDro:20}, we define a discretisation of $\Omega$ as a couple $\D = ( \M, \E)$, where:
\begin{itemize}
		\item the mesh $\M$ is a partition of $\Omega$, i.e.~$\M$ is a finite collection of disjoint, open, Lipschitz polytopes $K \subset \Omega$ with $|K|_d>0$ (the cells) such that $\overline{\Omega} = \bigcup _ {K \in \M } \overline{K}$; 
		\item the set $\E$ is a partition of the mesh skeleton $\partial\M=\bigcup_{K\in\M}\partial K$, i.e.~$\E$ is a finite collection of disjoint, connected, relatively open subsets $\s$ of $\overline{\Omega}$ with $|\s|_{d{-}1}>0$ (the faces) such that $\partial\M = \bigcup_{\s \in \E} \overline{\s}$. It is assumed that, for all $\s \in \E$, $\s$ is a Lipschitz polytopal subset of an affine hyperplane of $\R^d$.
We assume that, for all $K\in\M$, there exists a subset $\E_K$ of $\E$ (the set of faces of the cell $K$) such that $\partial K=\bigcup_{\s\in\E_K}\overline{\s}$. Finally, we let $n_{K,\s} \in \R^d$ be the (constant) unit normal vector to $\s \in \E_K$ pointing outward from $K$.	
\end{itemize}
The diameter of a subset $X \subset \overline{\Omega}$ is denoted by $h_X = \sup \lbrace|x-y| \mid (x,y) \in X^2  \rbrace$, and we define the size of $\D$ (the mesh size) as $h_\D = \max_{K\in\M} h_K$. For further use, we also introduce the smallest cell diameter $h_\flat=\min_{K\in\M}h_K$ of $\D$.

When studying asymptotic behaviours with respect to the mesh size, one has to adopt a measure of regularity for refined families of discretisations.
We classically follow~\cite[Definition 1.9]{DPDro:20}, in which regularity for a refined mesh family is quantified by a uniform (with respect to the mesh size) parameter $\theta\in(0,1)$, called mesh regularity parameter.
This parameter measures the chunkiness of the cells, but also the diameter ratio between the cells and their faces.
In what follows, to avoid the proliferation of multiplicative constants, we write $a\lesssim b$ in place of $Ca\leq b$ if $C>0$ only depends on $\Omega$, on the mesh regularity parameter $\theta$, and (if need be) on $\Lambda$, $\phi$, and the underlying polynomial degree, but is independent of both $h_{\D}$ (and $h_\flat$) and $\Delta t$.

\rev{
\begin{rem}[Relaxation of the mesh regularity assumptions]
  Upon replacing the scalings $h_\s$ for $\s\in\E_K$ by $h_K$ in the stabilisations/discrete norms below, the analysis performed in this work remains valid under the (much) less stringent mesh regularity assumptions of~\cite[Assumption 1]{DroYe:22} (cf.~also~\cite[Definition 1.41]{DPDro:20}). Contrary to~\cite[Definition 1.9]{DPDro:20}, these relaxed mesh regularity assumptions allow for small faces and cells with numerous faces, as they may appear in agglomeration-based meshing.
\end{rem}
}

Last, we make an additional regularity assumption on the anisotropy tensor. We assume that
\begin{equation} \label{eq:ass.dif.ten}
  \Lambda_{\mid K}\in W^{1,\infty}(K)^{d\times d}\qquad\forall K\in\M.
\end{equation}

\subsection{Discrete space and operators}

For $q\in\mathbb{N}$, and $X$ subset of $\overline{\Omega}$ of Hausdorff dimension $1\leq l\leq d$, we let $\poly^q(X)$ denote the vector space of $l$-variate polynomial functions $X \to \R$ of total degree at most $q$.
We also define the $L^2(X)$-orthogonal projector $\Pi_X^q : L^1(X) \to \poly^q(X)$ such that, given any $v \in L^1(X)$, $\Pi_X^q(v)$ is the only element in $\poly^q(X)$ satisfying 
\[
	\int_X \Pi_X^q(v) z = \int_X v z\qquad\forall  z \in \poly^q(X).
\]
Given any $K \in \M$, we also introduce the vector space $\poly^q(K)^d$ of $d$-variate polynomial vector fields $K\to\R^d$ of total degree at most $q$, as well as the corresponding $L^2(K)^d$-orthogonal projector (denoted as its scalar version) $\Pi_K^q:L^1(K)^d\to\poly^q(K)^d$. For any $\s \in \E_K$ and $v \in W^{1,1}(K)$, we also introduce the shortcut notation 
\[
	\Pi_\s^q(v) = \Pi_\s^q  \left (v_{\mid \s} \right ).
\]

Let $k$ be a given non-negative integer. We introduce the mixed-order HHO space (see~\cite{CErPi:21,DPDro:20}), with face unknowns of degree $k$ and (enriched) cell unknowns of degree $k+1$: 
\begin{equation*}
	\V_\D^{k}  = \left \lbrace \v_{\D} =\big( (v_K )_{K \in \M } ,
			(v_\s)_{\s \in \E} \big) 
			\left |
			\begin{array}{ll}
				\forall K\in\M, & v_K \in \poly^{k+1}(K)  \\
				\forall\s\in\E, & v_\s \in \poly^k(\s)   
			\end{array}				 	
			\right. 			
			\!\!\!\!\right \rbrace.
\end{equation*}
Given a cell $K \in \M$, we let
\[
\V_K^{k} = \poly^{k+1}(K) \times \bigg(\bigtimes_{\s \in \E_K} \poly^k(\s)\bigg)
\]
be the restriction of $\V_\D^{k}$ to $K$, and for a generic discrete element $\v_{\D} \in \V_{\D}^{k}$, we denote by $\v_K=\big(v_K,(v_{\s})_{\s\in\E_K}\big)\in\V_K^{k}$  its local restriction to the cell $K$.
To any $\v_{\D} \in \V_{\D}^{k}$, we associate the two piecewise polynomial functions $v_\M : \Omega \to \R$ and $v_\E : \partial\M  \to \R$ such that
\[
	{v_\M}_{|K} = v_K \text{ for all } K \in \M\quad \text{ and }\quad {v_\E}_{|\s} = v_\s \text{ for all } \s \in \E.
\] 
We also let $\one_\D \in \V_\D ^{k}$ be the discrete element such that $1_K = 1$ for all $K \in \M$ and $1_\s = 1$ for all $\s \in \E$.
Last, given a cell $K \in \M$, we define the local interpolator $\underline{I}^k_K : W^{1,1}(K) \to \V_K^{k}$ such that, for any $v \in W^{1,1}(K)$, 
\[
	\underline{I}^k_K (v) = \left ( \Pi_K^{k+1}(v) , \big( \Pi_\s^{k}(v)\big)_{\s \in \E_K} \right).
\]
Similarly, the global interpolator $\underline{I}^k_\D : W^{1,1}(\Omega) \to \V_\D^{k}$ is defined, for any $v \in W^{1,1}(\Omega)$, by
\[
	\underline{I}^k_\D (v) = \left ( 
		\big( \Pi_K^{k+1}(v_{\mid K}) \big)_{K \in \M},
		\big( \Pi_\s^{k}(v_{\mid\s})\big)_{\s \in \E} 
		\right).
                \]
                
As standard in the HHO context, locally to any cell $K\in\M$, we introduce a discrete gradient operator $G_K : \V_K^{k} \to \poly^k(K)^d$ such that, for any $\v_K \in \V_K^{k}$, $G_K(\v_K)\in\poly^k(K)^d$ satisfies
\begin{equation} \label{C4:def:grad}
	\int_K G_K(\v_K)  \cdot \tau = 
	-\int_K v_K \,\nabla \cdot \tau 
	+ \sum_{\s \in \E_K} \int_\s v_\s \,\tau \cdot n_{K,\s} \qquad \forall \tau \in  \poly^k(K)^d.
\end{equation}
This operator is a consistent discrete counterpart of the gradient operator.
It satisfies the following commutation property:
\[
	\forall v \in W^{1,1}(K), \qquad 
	G_K \circ \underline{I}_K^k (v) = \Pi_K^k(\nabla v).
\]
Given a face $\s \in \E_K$, we also define the jump operator $J_{K,\s} : \V_K^{k} \to \poly^k(\s)$ such that, for $\v_K\in\V_K^k$,
\begin{equation} \label{C4:def:jumps}
	J_{K,\s} (\v_K) = \Pi_\s^k(v_K) - v_\s .
\end{equation}
Based on the above ingredients, one can define an HHO counterpart of the local diffusion bilinear form $(z,v) \mapsto \int_K \Lambda \nabla z \cdot \nabla v$. We let $a_K : \V_K^{k} \times \V_K^{k} \to \R$ be the bilinear form such that
\begin{equation} \label{C4:sch:aK}
	a_K : (\z_K, \v_K) \mapsto
	\int_K \Lambda G_K(\z_K) \cdot  G_K(\v_K)
		+ \sum_{\s \in \E_K}  \frac{\Lambda_{K,\s}}{h_\s} 
			\int_\s J_{K,\s}(\z_K) J_{K,\s}(\v_K), 
\end{equation}
where $\Lambda_{K,\s} = \|\Lambda_{\mid K} n_{K,\s} \cdot n_{K,\s}  \|_{L^\infty(\s)}$ (recall the regularity assumption~\eqref{eq:ass.dif.ten}). 
In the context of HDG methods, the linear stabilisation used in~\eqref{C4:sch:aK} is often called Lehrenfeld--Sch\"oberl stabilisation, as it was first introduced in~\cite{Lehre:10,LehSc:16}.
Classically, one can then define a global bilinear form $a_\D : \V_\D^{k} \times \V_\D^{k} \to \R$, discretisation of $(z,v) \mapsto \int_\Omega \Lambda \nabla z \cdot \nabla v$, by summing the local contributions: 
\begin{equation}\label{C4:sch:aD}
	a_\D : (\z_\D, \v_\D) \mapsto
		\sum_{K \in \M} a_K (\z_K, \v_K).
\end{equation}

For analysis purposes, we introduce a discrete $H^1$-like semi-norm on $\V^k_\D$. Given a cell $K \in \M$, we first let, for any $\v_K\in\V^k_K$,
\[
	|\v_K|_{1,K}^2 = \Vert \nabla v_K \Vert_{L^2(K)^d}^2 + \sum_{\s \in \E_K} \frac{1}{h_\s} \Vert v_K-v_\s \Vert _{L^2(\s)}^2.
\] 
Then, at the global level, for any $\v_\D\in\V^k_\D$, we define
\begin{equation} \label{C4:def:norm}
	|\v_\D|^2_{1,\D} =  \sum_{K \in \M}|\v_K|_{1,K}^2.
\end{equation}
Notice that $|{\cdot}|_{1,\D}$ is not a norm on $\V_\D^{k}$, but any $\v_\D \in \V_\D^{k}$ satisfying $|\v_\D|_{1,\D} = 0$ is proportional to $\one_\D$.
In particular, this implies that $|{\cdot}|_{1,\D}$ is a norm on the zero-mass subspace of $\V_\D^{k}$ defined by 
\[
\V_{\D,0}^{k} = \left \lbrace \v_\D \in \V_\D^{k} \mid \int_\Omega v_\M = 0 \right \rbrace .
\] 
Standard HHO analysis implies the following stability estimate: 
\begin{equation} \label{C4:eq:coercivity}
	\forall \v_\D \in \V_{\D}^{k}, \qquad
		|\v_\D|_{1,\D}^2\lesssim a_\D(\v_\D,\v_\D),
\end{equation}
where the multiplicative constant is proportional to $\lambda_\flat$.
In particular, since $|{\cdot}|_{1,\D}$ is a norm on $\V_{\D,0}^{k}$, this estimate implies that $a_\D$ is coercive on $\V_{\D,0}^{k}$.
Finally, we recall the following discrete Poincar\'e--Wirtinger inequality (cf.~\cite[Theorem 6.5, $p=q=2$]{DPDro:20} in the equal-order case):
\begin{equation} \label{C4:eq:poinca}
	\forall \v_\D \in \V_{\D,0}^{k}, \qquad \|v_\M\|_{L^2(\Omega)} \lesssim |\v_\D|_{1,\D}.
\end{equation}
  
\subsection{Exponential fitting scheme}

The construction extends the ideas from~\cite{CHHLM:22}. Our (linear) scheme hinges on the exponential fitting strategy~\cite{BMaPi:89}.
In a nutshell, the exponential fitting approach is based on the following rewriting of the PDE flux: letting $\omega = \e^{-\phi}$, and introducing the Slotboom variable $\rho = u/\omega$, one has
\begin{equation}\label{C4:eq:sltotflux}
	\Lambda(\nabla u + u \nabla \phi) = \omega \Lambda\nabla \rho, 
\end{equation}
which allows to transform an advection-diffusion equation in $u$ into a purely diffusive, unconditionally coercive (by regularity of $\phi$, $\omega$ is a.e.~uniformly bounded away from zero) problem in $\rho$.
At the discrete level, the problem is solved in the Slotboom variable, which is sought in $\V_\D^{k}$. The discrete density is then \rev{defined} mimicking the relation $u = \omega \rho$. 

In view of~\eqref{C4:eq:sltotflux}, in order to define our exponential fitting (mixed-order) HHO scheme, we need to introduce a discrete counterpart of the bilinear form $(\rho, v ) \mapsto \int_\Omega \omega \Lambda \nabla \rho \cdot \nabla v $. 
To do so, given $K \in \M$, and leveraging the definition~\eqref{C4:sch:aK} of $a_K$, we let $a_K^\omega : \V_K^{k} \times \V_K^{k} \to \R$ be such that
\begin{equation} \label{eq:aK.om}
	a_K^\omega : (\ro_K, \v_K) \mapsto 
		\int_K \omega \Lambda G_K(\ro_K) \cdot  G_K(\v_K)
		+ \sum_{\s \in \E_K}  \frac{\Lambda_{ K,\s}^\omega}{h_\s} 
			\int_\s J_{K,\s}(\ro_K) J_{K,\s}(\v_K), 
\end{equation}
where $\Lambda_{ K,\s}^\omega = \|\omega \Lambda_{\mid K} n_{K,\s} \cdot n_{K,\s}  \|_{L^\infty(\s)}$.
At the global level, as previously, we construct the bilinear form $a_\D^\omega : \V_\D^{k} \times \V_\D^{k} \to \R$ by summing the local contributions: 
\begin{equation} \label{eq:aD.om}
	a_\D^\omega : (\ro_\D, \v_\D) \mapsto \sum_{K \in \M } a_K^\omega ( \ro_K, \v_K  ). 
\end{equation}
We can now introduce the exponential fitting HHO scheme for Problem~\eqref{C4:pb:evol}: 
find $(\ro_\D^n)_{n \geq 1} \in \left ( \V_\D^{k} \right ) ^{\mathbb{N}^\star}$ such that, for all $n \in \mathbb{N}$, 
\begin{subequations}\label{C4:sch:ExpF}
        \begin{empheq}[left = \empheqlbrace]{align}
        	\rev{\sum_{K\in\M} \int_K\frac{\mathfrak{u}^{\omega,n+1}_K - \mathfrak{u}^{\omega,n}_K}{\Delta t}  v_K} + a^\omega_\D(\ro_\D^{n+1}, \v_\D)
        		&=0 
        		&&\forall \v_\D \in \V_\D^{k}, \label{C4:sch:ExpF:test}\\        		
        	\mathfrak{u}^{\omega,n+1}_{K} &= \omega_{\mid K} \rho_K^{n+1}    	&&\forall K\in\M, \label{C4:sch:ExpF:u} \\
        		\mathfrak{u}^{\omega,0}_{K} &= u_{|K}^{0}  	&&\forall K\in\M. \label{C4:sch:ExpF:ini} 
        \end{empheq}
\end{subequations}
\rev{For any solution $(\ro_\D^n)_{n \geq 1}$ to~\eqref{C4:sch:ExpF}, we define a sequence of corresponding densities $(\ue_\D^{\omega,n})_{n \geq 1}$ as follows.
To the discrete Slotboom variable $\ro_\D \in \V_\D^{k}$ we associate the discrete density 
\[
	\ue_\D^\omega = \big( (\mathfrak{u}^\omega_K)_{K \in \M} , (\mathfrak{u}^\omega_\s)_{\s \in \E} \big),
\]
defined, consistently with~\eqref{C4:sch:ExpF:u}, as the collection of (a priori non-polynomial) functions
\begin{equation}\label{eq:rec.den}
  \mathfrak{u}^\omega_K = \omega_{\mid K} \rho_K\text{ for all }K \in \M
	\quad \text{ and } \quad 
	\mathfrak{u}^\omega_\s = \omega_{\mid\s} \rho_\s\text{ for all }\s \in \E.
\end{equation}
The non-polynomial nature of the components of $\ue_\D^\omega$ is, here and in what follows, emphasised by the use of Gothic fonts.
Finally, to any discrete density $\u^\omega_\D$, we associate the two piecewise smooth functions $\mathfrak{u}^\omega_\M : \Omega \to \R$ and $\mathfrak{u}^\omega_\E : \partial \M  \to \R$ such that
\begin{equation*}
  {\mathfrak{u}^\omega_\M}_{|K} = \mathfrak{u}^\omega_K \text{ for all } K \in \M\quad \text{ and }\quad {\mathfrak{u}^\omega_\E}_{|\s} = \mathfrak{u}^\omega_\s \text{ for all } \s \in \E.
\end{equation*}
Remark that $\mathfrak{u}^\omega_\M  = \omega\rho_\M$ in $\Omega$ and $\mathfrak{u}^\omega_\E = \omega_{\mid\partial\M}\rho_\E$ on $\partial\M$.
}
\begin{rem}[Non-polynomial \rev{definition}]\label{C4:rem:ExpF:nonpoly}
Here, we choose to \rev{define} a discrete density with (a priori) non-polynomial components.
One could also think of \rev{defining} a density with polynomial components, by multiplying (component by component) $\ro_\D$ by $\underline{I}^k_\D(\omega)$. This is how the solution to the low-order HFV exponential fitting scheme of~\cite{CHHLM:22} was defined (in that case, both cell/face unknowns were constants).
\end{rem}
\begin{rem}[Initial condition] \label{C4:rem:ExpF:ini}
Notice that we do not define $\mathfrak{u}^{\omega,0}_\M$ in the same way as $\mathfrak{u}^{\omega,n+1}_{\M }$.
Indeed, we directly use in the definition~\eqref{C4:sch:ExpF:ini} the initial datum $u^{0}$.
\rev{This choice is motivated by the following observation. Define $\rho^0=u^0/\omega$, and let $\mathfrak{u}^{\omega,0}_K=\omega_{\mid K}\Pi^{k+1}_K(\rho^0_{\mid K})$ for all $K\in\M$ in place of~\eqref{C4:sch:ExpF:ini}. Then, the resulting discrete solution's mass is $M^\omega = \sum_{K \in \M} \int_K \mathfrak{u}^{\omega,0}_K \neq M$ in general.
} 
\end{rem}

Recall that $\int_{\Omega}u^0=M>0$. Let $\ro_\D^\infty \in \V_\D^{k}$ be defined as
\[
	\ro_\D^\infty = c^M_{l} \one_\D, 
	\qquad\text{ with }
	c^M_l = \dfrac{M}{\int_\Omega \e^{-\phi}}>0. 
\]
One can easily check that $ \ro_\D^\infty$ is the only steady solution to the exponential fitting scheme~\eqref{C4:sch:ExpF}.
Based on~\eqref{eq:rec.den}, $ \ro_\D^\infty$ is associated to the discrete equilibrium density $\ue_\D^{\omega,\infty}$ such that
\begin{equation} \label{eq:di.th.eq.l}
    \mathfrak{u}^{\omega,\infty}_{\M}=c^M_l\e^{-\phi}\text{ in }\Omega\quad\text{ and }\quad\mathfrak{u}^{\omega,\infty}_{\E}=c^M_l\e^{-\phi_{\mid\partial\M}}\text{ on }\partial\M.
\end{equation}
\rev{It follows that the reconstructed discrete equilibrium density $\mathfrak{u}^{\omega,\infty}_\M$ (always) coincides with the thermal equilibrium~\eqref{C4:eq:def:equi} in $\Omega$. Such a striking property is, however, to be tempered by Remark~\ref{rem:alt} below.}
Following Remark~\ref{C4:rem:ExpF:nonpoly}, notice that if we had adopted instead a polynomial \rev{definition} for the discrete densities, we would have obtained that $\ue_\D^{\omega,\infty}=\underline{I}^k_\D(u^\infty)\in\V^k_\D$, as was the case for the low-order exponential fitting HFV method of~\cite{CHHLM:22} (in that case, both cell/face unknowns were constants). A drawback of such a \rev{definition}, compared to~\eqref{eq:rec.den}, is that the components of $\underline{I}^k_\D(u^\infty)$ are not necessarily positive functions (note that this issue does not exist in the low- and equal-order HFV case).

\begin{rem}[Alternative scheme definition] \label{rem:alt}
  \rev{Let $\phid_\D=\underline{I}_{\D}^k(\phi)\in\V_\D^k$.
  Another definition of the exponential fitting HHO scheme consists in replacing $\omega$ by $\e^{-\phi_\M}$ in the expressions of both $a_K^{\omega}$ (see~\eqref{eq:aK.om}) and $\mathfrak{u}_K^{\omega,n+1}$ (see~\eqref{C4:sch:ExpF:u}).
  Then, in place of~\eqref{eq:rec.den}, the following definition of discrete densities is adopted:
  \begin{equation}\label{eq:rec.den.alt}
  \mathfrak{u}^\omega_K = \e^{-\phi_K} \rho_K\text{ for all }K \in \M
	\quad \text{ and } \quad 
	\mathfrak{u}^\omega_\s = \e^{-\phi_\s} \rho_\s\text{ for all }\s \in \E.
  \end{equation}
  Such a scheme is somewhat closer to what one would encounter in the context of semiconductor models, since $\phi$ would be unknown, and sought, at the discrete level, in $\V_\D^k$.} In this case, the discrete equilibrium density $\ue_\D^{\omega,\infty}$ would satisfy, in place of~\eqref{eq:di.th.eq.l}, the same kind of relations as~\eqref{eq:di.th.eq.nl}--\eqref{eq:di.th.eq.nl.2} below. Here, we \rev{rather} choose to exploit the full knowledge we have of the potential $\phi$ to define the scheme.
\end{rem}

\subsection{Nonlinear scheme}

The construction extends the ideas from~\cite{CHHLM:22,Moatt:23SC}. Our nonlinear scheme relies on a nonlinear reformulation of Problem \eqref{C4:pb:evol}~\cite{CaGui:17}. To do so, we introduce the logarithm potential $\ell = \log(u)$ and the quasi-Fermi potential $w = \ell + \phi$. At least formally, if $u$ is positive, one has the following relation on the PDE flux:
\begin{equation} \label{C4:eq:nlflux}
  \Lambda\big(\nabla u + u \nabla \phi\big) = u \Lambda\nabla \left ( \log(u) + \phi \right )= \e^\ell  \Lambda\nabla w.
\end{equation}
We choose to discretise the potentials as piecewise polynomials, i.e.~we approximate $\ell$ and $w$ as discrete unknowns in $\V_\D^{k}$. Then, mimicking the relation $u = \e^\ell$, each discrete density component is \rev{defined} as the exponential of a polynomial, thus ensuring its positivity. 
 
In view of~\eqref{C4:eq:nlflux}, in order to define our nonlinear HHO scheme, we shall introduce a discrete counterpart of the map $(\ell;w,v)\mapsto \int_\Omega \e^\ell \Lambda \nabla w \cdot \nabla v$. Locally to any cell $K\in\M$, this discrete counterpart is built as the sum of a consistent~\eqref{C4:sch:cons} and a stabilising~\eqref{C4:sch:stab_nl} contributions: for all $\l_K,\w_K,\v_K\in\V_K^k$, we let
\begin{subequations}\label{C4:sch:loc}
        \begin{empheq}{align}
        	\mathcal{C}_K(\l_K;\w_K, \v_K) 
        		&= \int_K \e^{\ell_K} \Lambda \,G_K(\w_K) \cdot G_K(\v_K), \label{C4:sch:cons} \\
 		\mathcal{S}_K(\l_K; \w_K, \v_K) 
 			&=  \sum_{\s \in \E_K } \frac{\Lambda_{K,\s}}{h_\s} 
				\int_\s  \frac{ \e^{\Pi_\s^k(\ell_K)} + \e^{ \ell_\s}  }{2}
	 			J_{K,\s}(\w_K) J_{K,\s}(\v_K).   \label{C4:sch:stab_nl} 
        \end{empheq}
\end{subequations}
We then introduce the local map $\mathcal{T}_K :  \V_K^{k} \times \V_K^{k} \times \V_K^{k} \to \R$  such that 
\begin{equation}
	\mathcal{T}_K :(\l_K;\w_K, \v_K)\mapsto 
		\mathcal{C}_K(\l_K;\w_K, \v_K) + \mathcal{S}_K(\l_K; \w_K, \v_K)
		+  \varepsilon\, h_K^{k+2} a_K(\w_K, \v_K),
\label{C4:sch:stationnary}
\end{equation}
where $\varepsilon$ is a non-negative parameter and $a_K$ is the bilinear form defined by~\eqref{C4:sch:aK}.
At the global level, we finally define the map $\mathcal{T}_\D :  \V_\D^{k} \times \V_\D^{k} \times \V_\D^{k} \to \R$  by summing the local contributions: 
\begin{equation}
	\mathcal{T}_\D: (\l_\D;\w_\D, \v_\D) \mapsto 
		 \sum_{K \in \M} \mathcal{T}_K (\l_K;\w_K, \v_K).
\label{C4:sch:stationnary:global}
\end{equation}
\begin{rem}[Parameter $\varepsilon$]
The map $\mathcal{T}_\D$ is to be understood as a discretisation of $(\ell;w,v) \mapsto \int_\Omega (\e^\ell + \epsilon)\, \Lambda\nabla  w \cdot \nabla v $, with $\epsilon$ of magnitude $\varepsilon h_\D^{k+2}$.
At the theoretical level, this $\epsilon$-perturbation of the model is necessary, at the moment, to show the existence result of Theorem~\ref{C4:th:nl:exist}. From a more practical viewpoint, the sensitivity of the method with respect to $\epsilon$ is not completely understood yet. First numerical experiments tend to show that the choices $\varepsilon = 1$ and $\varepsilon = 0$ produce essentially similar results. Concerning the choice of scaling factor $h_K^{k+2}$ in~\eqref{C4:sch:stationnary}, it seems to yield in practice (when $\varepsilon>0$) the expected orders of convergence.
The influence of the $\epsilon$-term will be further investigated in future works.
\end{rem}
\noindent
Let $\phid_\D=\underline{I}_{\D}^k(\phi)\in\V_\D^k$. We can now introduce our nonlinear HHO scheme for Problem~\eqref{C4:pb:evol}: find $\left (\l_\D^n \right)_{n \geq 1} \in \left (\V_\D^{k} \right ) ^\mathbb{N^\star}$ such that, for all $n \in \mathbb{N}$,
\begin{subequations}\label{C4:sch:AD}
        \begin{empheq}[left = \empheqlbrace]{align}
        	\rev{\sum_{K\in\M}\int_K \frac{\mathfrak{u}^{n+1}_K - \mathfrak{u}^n_K}{\Delta t} v_K} +  \mathcal{T}_\D(\l_\D^{n+1}; \l_\D^{n+1} +  \phid_\D, \v_\D)
        		&=0 
        		&&\forall \v_\D \in \V_\D^{k}, \label{C4:sch:test}\\        		
        		\mathfrak{u}^{n+1}_K &=	\e^{\ell_K^{n+1}}    	&&\forall K\in\M, \label{C4:sch:u} \\
        		\mathfrak{u}^0_K &= u^{0}_{|K}    	&&\forall K\in\M. \label{C4:sch:ini} 
        \end{empheq}
\end{subequations}
\rev{For any solution $\left (\l_\D^n \right)_{n \geq 1}$ to~\eqref{C4:sch:AD}, we define a sequence of corresponding positive densities $\left (\u_\D^n \right ) _{n \geq 1}$ as follows.
To the discrete logarithm potential $\l_\D \in \V_\D^{k}$ we associate the discrete density 
\[
 \u_\D = \big( (\mathfrak{u}_K)_{K \in \M}, (\mathfrak{u}_\s)_{\s \in \E}  \big), 
\]
defined, consistently with~\eqref{C4:sch:u}, as the collection of positive (non-polynomial) functions 
\begin{equation} \label{eq:rec.pos.den}
	\mathfrak{u}_K = \e^{\ell_K}\text{ for all }K \in \M 
	\quad \text{ and } \quad 
	\mathfrak{u}_\s = \e^{\ell_\s}\text{ for all }\s \in \E.
\end{equation}
The non-polynomial nature of the components of $\u_\D$ is, here also, emphasised by the use of Gothic fonts.
Finally, to any discrete density $\u_\D$, we associate the two (positive) piecewise smooth functions $\mathfrak{u}_\M : \Omega \to \R^\star$ and $\mathfrak{u}_\E : \partial\M  \to \R^\star$ such that
\begin{equation*} 
  {\mathfrak{u}_\M}_{|K} = \mathfrak{u}_K \text{ for all } K \in \M\quad \text{ and }\quad {\mathfrak{u}_\E}_{|\s} = \mathfrak{u}_\s \text{ for all } \s \in \E.
\end{equation*}
Remark that $\mathfrak{u}_\M =  \e^{\ell_\M}$ in $\Omega$ and $\mathfrak{u}_\E =  \e^{\ell_\E}$ on $\partial\M$.}

\begin{rem}[Initial condition] \label{C4:rem:nl:ini}
  Remark that we do not define $\mathfrak{u}_\M^0$ in the same way as $\mathfrak{u}_\M^{n+1}$. We indeed directly use in the definition~\eqref{C4:sch:ini} the initial datum $u^0$.
This strategy allows one to circumvent the definition of some $\ell_\M^0$, cell interpolate of $\ell^{0} = \log(u^{0})$, the latter quantity being undefined in regions where $u^{0}$ vanishes.
The question of defining an initial discrete logarithm potential remains however a major difficulty when it comes to numerical implementation, since it is needed for the initialisation of the Newton method when $n=0$ (see Section~\ref{C4:sec:num:tricks}).
\end{rem}

Recall that $\int_{\Omega}u^0=M>0$. Let $\l_\D^\infty\in\V^k_\D$ be defined as
\[
\l_\D^\infty=\log(c^M_{nl})\one_\D-\phid_\D,
\qquad\text{ with }
c^M_{nl}=\frac{M}{\int_\Omega \e^{-\phi_\M}}>0.
\]
It can be easily checked that $\l_\D^\infty$ is the only steady solution to the nonlinear scheme~\eqref{C4:sch:AD}. Based on~\eqref{eq:rec.pos.den}, $\l_\D^\infty$ is associated to the discrete equilibrium density 
 $\u_\D^\infty$ such that
\begin{equation} \label{eq:di.th.eq.nl}
	\mathfrak{u}^\infty_\M = c^M_{nl} \e^{-\phi_\M}\text{ in } \Omega
	\quad \text{ and } \quad 
	\mathfrak{u}^\infty_\E = c^M_{nl} \e^{-\phi_\E}\text{ on }\partial\M. 
\end{equation}
In a sense we make clear just below, the discrete equilibrium density $\u_\D^\infty$ is a reasonable approximation of the thermal equilibrium $u^\infty$ defined by~\eqref{C4:eq:def:equi}:
\begin{subequations} \label{eq:di.th.eq.nl.2}
  \begin{empheq}{align}
  \log(\mathfrak{u}^\infty_{K}) - \Pi_K^{k+1}\big( \log(u^\infty)\big) 
  &= \log \left(\int_\Omega \e^{-\phi} \right )  - \log \left(\int_\Omega \e^{-\phi_\M} \right )&\;&\forall K\in\M,\\
  \log(\mathfrak{u}^\infty_{\s}) - \Pi_\s^{k}\big( \log(u^\infty)\big) 
  &= \log \left(\int_\Omega \e^{-\phi} \right )  - \log \left(\int_\Omega \e^{-\phi_\M} \right )&\;&\forall \s\in\E.
  \end{empheq}
\end{subequations}
\rev{It follows that the reconstructed discrete equilibrium density $\mathfrak{u}^{\infty}_\M$ satisfies: if $\phi_{\mid K}\in\poly^{k+1}(K)$ for all $K\in\M$, then $\mathfrak{u}_\M^\infty=u^\infty$ in $\Omega$.}
Remark that the discrete equilibrium density $\u_\D^\infty$ is not equal to $\underline{I}_{\D}^k(u^\infty)\in\V_\D^k$.
This is in contrast with what held true for the low-order nonlinear HFV method of~\cite{CHHLM:22} (in that case, both cell/face unknowns were constants). This can be explained by the choice of discretisation for $\phi$, which was taken as $-\log(\underline{\omega}_\D)$ (with $\underline{\omega}_\D$ HFV interpolate of $\omega$) in~\cite{CHHLM:22} in place of $\phid_\D$ here (the latter choice is inspired by~\cite{Moatt:23SC} in the context of semiconductors). Remark that, in practice, $-\log(\underline{\omega}_\D)$ and $\phid_\D$ may coincide if the integrals are approximated using an evaluation at the barycenter. This was the case in the numerical experiments of~\cite{CHHLM:22}.

\section{Main features of the schemes} \label{C4:sec:prop}

We present in this section some theoretical results about the two schemes introduced above.
We focus, in particular, on the existence (and stability) of solutions, as well as on questions related to their long-time behaviour.
The results presented below generalise those obtained in~\cite{CHHLM:22} in the low-order HFV context. 
In particular, the analysis strongly hinges on the entropy structure of both schemes.

\begin{rem}[Lowest-order versions of the schemes ($k=0$)] \label{C4:rem:k=0}
  Note that the lowest-order versions of the two schemes introduced above do not coincide with the exponential fitting and nonlinear HFV schemes of~\cite{CHHLM:22}.
  \rev{Indeed,} the lowest-order versions of the methods~\eqref{C4:sch:ExpF} and~\eqref{C4:sch:AD} make use of (enriched) affine cell unknowns, whereas HFV schemes use constants.
  \rev{Also, whereas the nonlinear HFV method is built upon a stable discrete gradient operator (defined on a pyramidal submesh), the present nonlinear scheme is defined following the standard HHO philosophy of splitting consistency and (nonlinear) stabilisation.}
  Therefore, the results presented here are new, even for $k=0$. 
\end{rem}
Before presenting individual results for each scheme, let us stress that both schemes exhibit a similar important property: the preservation of the thermal equilibrium.
\rev{
\begin{prop}[Preservation of the thermal equilibrium] \label{prop:th.eq}
  The alternative (fully discrete) exponential fitting scheme of Remark~\ref{rem:alt} and the nonlinear scheme~\eqref{C4:sch:AD} preserve the thermal equilibrium in the sense of Definition~\ref{def:th.eq}. More precisely,
  \begin{itemize}
    \item[$\bullet$] $\log(\u^{\omega,\infty}_\D)+\phid_\D$ is proportional to $\one_\D$;
    \item[$\bullet$] $\l^\infty_\D + \phid_\D$ is proportional to $\one_\D$.
  \end{itemize}
\end{prop}
\noindent
For the original exponential fitting scheme~\eqref{C4:sch:ExpF}, there holds that $\log(\mathfrak{u}^{\omega,\infty}_\M) + \phi$ over $\Omega$ and $\log(\mathfrak{u}^{\omega,\infty}_\E) + \phi_{\mid\partial\M}$ over $\partial\M$ are equal to the same constant.}

\subsection{Exponential fitting scheme}

We present here the main properties of the exponential fitting HHO scheme~\eqref{C4:sch:ExpF}, and give detailed proofs of the results.
As a preliminary remark, note that since $\phi\in \rev{W^{1,\infty}(\Omega)}$, 
\begin{equation} \label{eq:phi:Linf}
  \|\phi\|_{L^\infty(\Omega)}\lesssim 1.
\end{equation}
As a consequence, the tensor $\omega \Lambda$ is a.e.~uniformly elliptic.
Recalling~\eqref{C4:eq:coercivity}, this implies the following stability estimate:
\begin{equation} \label{C4:eq:omega:coercivity}
	\forall \v_\D \in \V_{\D}^{k}, \qquad 	
		|\v_\D|_{1,\D}^2\lesssim a_\D^\omega(\v_\D,\v_\D),
\end{equation}
where the multiplicative constant is proportional to $\lambda_\flat$.
We first state a well-posedness result, which is mainly a consequence of the previous stability estimate.

\begin{restatable}[Well-posedness of the exponential fitting scheme]{prop}{propExpFWP} \label{C4:prop:ExpF:exi}
The exponential fitting scheme~\eqref{C4:sch:ExpF} admits a unique solution $(\ro_\D^n)_{n\geq 1}$. 
Moreover, the corresponding discrete densities $(\ue_\D^{\omega,n})_{n \geq 1}$ have a mass equal to $M$: 
\begin{equation} \label{C4:eq:ExpF:mass}
	\forall n \geq 1, \qquad 
		\int_\Omega \mathfrak{u}_\M^{\omega,n} = \int_\Omega u^{0} = M.
\end{equation}
\end{restatable}
\begin{proof}
Let $n \geq 0$, and assume that $\mathfrak{u}_\M^{\omega, n}$ is defined. We want to show that equations~\eqref{C4:sch:ExpF:test}--\eqref{C4:sch:ExpF:u} admit a unique solution.
To do so, we first define, for any $\ro_\D \in \V_\D^{k}$,
\[ 
	\| \ro_\D \|^2_{\omega , \Delta t, \D} = \Delta t |\ro_\D|_{1,\D}^2 + \int_\Omega \omega \rho_\M^2.
\]
Since $|{\cdot}|_{1,\D}$ is a semi-norm on $\V_\D^{k}$ with zero set spanned by $\one_\D$, it follows that the map $\|{\cdot}\| _{\omega , \Delta t, \D}$ defines a norm on $\V_\D^{k}$. 
Thus, by~\eqref{C4:eq:omega:coercivity}, the bilinear form $A_\D^\omega: (\ro_\D, \v_\D) \mapsto \int_\Omega\omega \rho_\M v_\M  + \Delta t a_\D^\omega(\ro_\D, \v_\D)$ satisfies the following coercivity property:
\[
	\forall \ro_\D \in \V_\D^{k}, \qquad 
		\| \ro_\D \| _{\omega , \Delta t, \D} ^2\lesssim A_\D^\omega(\ro_\D,\ro_\D ).
\] 
By the Lax--Milgram lemma, equations~\eqref{C4:sch:ExpF:test}--\eqref{C4:sch:ExpF:u} therefore admit a unique solution $\ro_\D^{n+1}$ in $\V_\D^{k}$, from which one can uniquely define $\ue^{\omega,n+1}_\D$ by~\eqref{eq:rec.den}.
To prove mass conservation, we just test~\eqref{C4:sch:ExpF:test} by $\one_\D$ to get
\[
		\int_\Omega \frac{\mathfrak{u}^{\omega,n+1}_\M - \mathfrak{u}^{\omega,n}_\M}{\Delta t} = 0.
\]
We conclude by noticing that $\int_\Omega\mathfrak{u}^{\omega,0}_\M = \int_\Omega u^{0}$ according to~\eqref{C4:sch:ExpF:ini}.
\end{proof}

We now state our main result about the exponential fitting scheme, which ensures that the solution to~\eqref{C4:sch:ExpF} has similar long-time behaviour as the PDE solution.
As usual with the entropy method, the main idea is to get a control of the entropy by its dissipation. Here, such an estimate is a consequence of the discrete Poincaré inequality~\eqref{C4:eq:poinca}.

\begin{restatable}[Long-time behaviour of the exponential fitting scheme]{prop}{propExpFLT}\label{C4:prop:ExpF:time}
Assume that $u^0\in L^2(\Omega)$. Let $(\ro_\D^n)_{n \geq 1}$ be the solution to the exponential fitting scheme~\eqref{C4:sch:ExpF}. Then, the following discrete entropy relation holds true:
\begin{equation} \label{C4:eq:ExpF:entro}
	\forall n \in \mathbb{N}, \qquad \frac{\Entro^{n+1}_\omega -\Entro^{n}_\omega }{\Delta t} +\Diss_\omega^{n+1} \leq 0, 
\end{equation}
where the discrete quadratic entropy is defined as
\[
\Entro^{n}_\omega = \frac{1}{2}\int_\Omega \omega ( \rho_\M^n - \rho_\M^\infty)^2
\]
with $\rho_\M^0$ defined (with a slight abuse in notation, since $\rho_\M^0$ is not piecewise polynomial) by $\rho_\M^0=\frac{u^{0}}{\omega}$,
and the discrete dissipation is given by 
\[
	\Diss_\omega^n = a_\D^\omega (\ro_\D^n -\ro_\D^\infty,\ro_\D^n -\ro_\D^\infty ) \qquad \forall n \geq 1.
        \]
As a consequence, the \rev{reconstructed} discrete density converges exponentially fast in time towards the \rev{reconstructed} discrete equilibrium density:
there exists a positive constant $\nu_\omega$, independent of both $h_\D$ and $\Delta t$, such that
\begin{equation} \label{C4:eq:ExpF:LT}
	\forall n \in \mathbb{N},\qquad 
		\|\mathfrak{u}_\M^{\omega,n} - \mathfrak{u}_\M^{\omega,\infty} \|_{L^2(\Omega)} 
		\lesssim  \left (  1 + \nu_\omega \Delta t \right )^{-\frac{n}{2}} 
			\|u^{0}  -\mathfrak{u}_\M^{\omega,\infty} \|_{L^2(\Omega)}.
\end{equation}	 
\end{restatable}
\begin{proof}
Let $n \in \mathbb{N}$. By convexity of $x\mapsto x^2$ on $\R$, one has 
\[
	\Entro^{n+1}_\omega - \Entro^n_\omega  = 
		\frac{1}{2} \int_\Omega \omega \left ( (\rho_\M^{n+1} - \rho_\M^\infty )^2 - (\rho_\M^{n} - \rho_\M^\infty )^2 \right )
		\leq \int_\Omega \omega \left (\rho_\M^{n+1} - \rho_\M^{n}\right ) \left ( \rho_\M^{n+1} - \rho_\M^\infty \right ).  
\]
Therefore, testing~\eqref{C4:sch:ExpF:test} against $\ro_\D^{n+1} - \ro_\D^\infty \in \V_\D^{k}$, we get
\[
	\frac{\Entro^{n+1}_\omega - \Entro^n_\omega}{\Delta t} \leq - a_\D^\omega ( \ro_\D^{n+1} , \ro_\D^{n+1} - \ro_\D^\infty).
\]
Note that this estimate holds true also for $n=0$ (using the definition of $\rho_\M^{0}$).
On the other hand, by the expression of $\ro_\D^\infty$ (proportional to $\one_\D$), $a_\D^\omega ( \ro_\D^{\infty} , \ro_\D^{n+1} - \ro_\D^\infty) = 0$, hence by bilinearity of $a_\D^\omega$, 
\[
\Diss_\omega^{n+1}=a_\D^\omega ( \ro_\D^{n+1} , \ro_\D^{n+1} - \ro_\D^\infty),
\]
which yields the entropy relation~\eqref{C4:eq:ExpF:entro}. To get the exponential decay, one needs to compare $\Diss_\omega^{n+1}$ with $\Entro^{n+1}_\omega$.
To do so, we let $\v_\D =  \ro_\D^{n+1} - \ro_\D^\infty \in \V_\D^{k}$, and we define the probability measure $\dd\mu = \frac{c^M_l}{M}\omega \dd x $ on $\Omega$.
We define $\langle v_\M\rangle_\mu$ as the mass of $v_\M$ for the measure $\dd\mu$, i.e.
\[
	\langle v_\M\rangle_\mu = \int_\Omega v_\M \dd \mu.
\]
The definition~\eqref{eq:rec.den}, and the mass preservation identity~\eqref{C4:eq:ExpF:mass}, imply that 
\[
	\langle v_\M\rangle_\mu = \int_\Omega v_\M \dd \mu = \frac{c^M_{l}}{M} \int_\Omega \omega v_\M
	= \frac{c^M_{l}}{M} \int_\Omega \left ( \mathfrak{u}_\M^{\omega, n+1} - \mathfrak{u}_\M^{\omega, \infty} \right )
	= c^M_{l}\frac{M - M} {M} = 0.
\]
Therefore, letting $\langle v_\M\rangle = \frac{1}{|\Omega|_d} \int_\Omega v_\M$, and applying~\cite[Lemma 5.2, $q=2$]{CCHHK:20}, we get 
\[
	\frac{c^M_{l}}{M}\int_\Omega \omega  v_\M ^2 = \int_\Omega \left (  v_\M - \langle v_\M\rangle_\mu \right ) ^2 \dd \mu
	\leq 4 \int_\Omega \left ( v_\M  - \langle v_\M\rangle \right ) ^2 \dd \mu.
\]  
Using the definition of $\dd\mu$, and the bound~\eqref{eq:phi:Linf}, yields
\[
	\int_\Omega \omega v_\M ^2
		\lesssim \int_\Omega \left ( v_\M  -\langle v_\M\rangle  \right ) ^2
		 = \|  v_\M  -\langle v_\M\rangle  \|^2_{L^2(\Omega)}.
\]
By definition of $\langle v_\M\rangle$, one has $\v_\D - \langle v_\M\rangle\one_\D \in \V_{\D,0}^{k}$, so we can apply the discrete Poincaré--Wirtinger inequality~\eqref{C4:eq:poinca} to infer that 
\[
	\|  v_\M  - \langle v_\M\rangle \|^2_{L^2(\Omega)} \lesssim |\v_\D - \langle v_\M\rangle\one_\D |_{1,\D}^2
		= |\v_\D  |_{1,\D}^2.
\]
Combining the two previous estimates, we get 
\[
	\Entro_\omega ^{n+1}=\frac{1}{2}\int_\Omega\omega v_\M^2 \lesssim |\v_\D  |_{1,\D}^2.
\]
Now, one can use the stability estimate~\eqref{C4:eq:omega:coercivity} to infer that 
\[
	|\v_\D |_{1,\D}^2\lesssim a^\omega_\D ( \v_\D, \v_\D )=\Diss^{n+1}_\omega.
\]      
Therefore, combining the last two estimates, one infers the existence of $\nu_\omega>0$, independent of both $h_\D$ and $\Delta t$, such that the following relation between the entropy and its dissipation holds true:
\[
	\nu_{\omega}\Entro_\omega ^{n+1} \leq \Diss^{n+1}_\omega.
\]        
Plugging this estimate into the entropy relation~\eqref{C4:eq:ExpF:entro}, we deduce that
\[
	(1 + \nu_\omega \, \Delta t ) \Entro_\omega ^{n+1} \leq \Entro_\omega ^{n}. 
\]
This implies the exponential decay of the entropy: 
\[
	\forall n\geq 0, \qquad \Entro_\omega^n \leq \left ( 1 + \nu_\omega \, \Delta t\right )^{-n} \Entro_\omega^0.
\]
To conclude, we just use the definition~\eqref{eq:rec.den} and the bound~\eqref{eq:phi:Linf} to infer that
\[
	\|\mathfrak{u}^{\omega,n}_\M - \mathfrak{u}^{\omega,\infty}_\M  \|_{L^2(\Omega)}^2
	 \lesssim \Entro_\omega^n \lesssim 
	 	\|\mathfrak{u}^{\omega,n}_\M - \mathfrak{u}^{\omega,\infty}_\M  \|_{L^2(\Omega)}^2,
\]
which, combined with the fact that $\mathfrak{u}^{\omega,0}_\M=u^0$, finally yields~\eqref{C4:eq:ExpF:LT}.
\end{proof}

\begin{rem}[Regularity of the initial datum and topology of the convergence]
  Notice that in Proposition~\ref{C4:prop:ExpF:time} we have made the extra assumption that $u^{0} \in L^2(\Omega)$.
  The long-time analysis of the exponentially fitted model indeed relies on the decay of the quadratic entropy (in the Slotboom variable)
  \[ 	\Entro_\omega(t) = \frac{1}{2} \int_\Omega \omega (\rho(t) - \rho^\infty)^2 . \]
  In order to guarantee that the initial quadratic entropy is finite, assuming that the initial datum is in $L^2(\Omega)$ is a safe choice.
  At the end, as a reminiscence of the linearity of the model, the exponential fitting approach gives a convergence (in time) result in the $L^2$-topology (in space). 
In contrast, the nonlinear approach will yield convergence in a weaker norm (typically $L^1$), but can be used to deal with less regular initial data, which are in $L \log(L)$ only.
\end{rem}

\subsection{Nonlinear scheme}

We present here some results regarding the analysis of the nonlinear HHO method~\eqref{C4:sch:AD}. 
Since we deal with a nonlinear scheme, unlike the exponential fitting scheme, the question of the existence of solutions is the main difficulty here. As often for this type of method, we start by establishing some a priori estimates.
For the purpose of analysis, given a discrete logarithm potential $\l_\D \in \V_\D^{k}$, we associate a discrete quasi-Fermi potential $\w_\D\in\V^k_\D$ defined by 
\begin{equation} \label{eq:qua.fer}
	\w_\D = \l_\D + \phid_\D - \log(c^M_{nl}) \one_\D,
\end{equation}
where we recall that $c^M_{nl}=M/\int_\Omega\e^{-\phi_\M}$.
By \rev{definition~\eqref{eq:rec.pos.den}} and equation~\eqref{eq:di.th.eq.nl}, one has
\[
  w_\M = \log \left ( \frac{\mathfrak{u}_\M}{\mathfrak{u}_\M^\infty} \right)\text{ in }\Omega \quad\text{ and }\quad w_\E = \log \left ( \frac{\mathfrak{u}_\E}{\mathfrak{u}_\E^\infty} \right)\text{ on }\partial\M.
\]
Note that, on the other hand, for any $\l_\D,\v_\D\in \V_\D^{k}$, we have 
\begin{equation} \label{eq:TD}
  \mathcal{T}_\D(\l_\D; \l_\D +  \phid_\D, \v_\D) = \mathcal{T}_\D(\l_\D; \w_\D, \v_\D),
\end{equation}
since $(\l_\D +  \phid_\D)-\w_\D$ is proportional to $\one_\D$. 
Similarly to previous works on nonlinear HFV schemes for semiconductor models~\cite{Moatt:23SC}, the discrete quasi-Fermi potentials are the key variables to perform the analysis of the method.
As a last remark, notice that since $\phi\in H^1(\Omega)$, by boundedness of the interpolator (cf.~\cite[Proposition \rev{5.3}]{DPDro:20}), 
\begin{equation} \label{C4:eq:bounds:phi}
	|\phid_\D |_{1,\D} \lesssim 1.
\end{equation}

Let us now state some fundamental a priori relations. 
As for the exponential fitting scheme, the discrete entropy structure of the nonlinear scheme mainly results from the convexity of the entropy. 
\begin{restatable}[Fundamental a priori relations]{prop}{propnlapriori} \label{C4:prop:nl:apriori}
Let $\left (\l_\D^n \right)_{n \geq 1}$ be a given solution to the nonlinear scheme~\eqref{C4:sch:AD}, and $\left (\u_\D^n \right ) _{n \geq 1}$ be the corresponding discrete density. 
Then, the following a priori relations hold true:
\begin{enumerate}
	\item[(i)] the mass is preserved along time: 
\begin{equation} \label{C4:eq:nl:mass}
	\displaystyle \forall n \geq 1, \qquad \int_\Omega \mathfrak{u}^n_\M = \int_\Omega u^{0} = M;
\end{equation}
	\item[(ii)] a discrete entropy/dissipation relation is satisfied:
\begin{equation}\label{C4:eq:nl:entrorel}
	\displaystyle	\forall n \in \mathbb{N}, \qquad \frac{\Entro^{n+1} - \Entro^n }{\Delta t} + \Diss^{n+1}\leq 0, 
\end{equation}
where the discrete entropy and dissipation are non-negative quantities defined by
\[
	\Entro^n = \int_\Omega \mathfrak{u}^\infty_\M \Phi_1 \left ( \frac{\mathfrak{u}^n_\M}{\mathfrak{u}^\infty_\M}\right )
	\qquad \text{ and }  \qquad 
	\Diss^{n} = \mathcal{T}_\D (\l^n_\D;\w^n_\D, \w^n_\D)\,\text{ for }n\geq 1,
\]
with $\Phi_1:s \mapsto s \log(s) - s +1$ (and $\Phi_1(0) = 1$).
\end{enumerate}
\end{restatable}
\begin{proof}
Let $n\geq 0$.
Using $\one_\D$ as a test function in \eqref{C4:sch:test}, we get that the mass is conserved:
\[
	\int_\Omega \mathfrak{u}^{n+1}_\M = \int_\Omega \mathfrak{u}^{n}_\M.
\] 
Therefore, by~\eqref{C4:sch:ini}, we infer~\eqref{C4:eq:nl:mass}.
To establish the entropy relation, we first use the convexity of $\Phi_1$, which yields
\[	
	\Entro^{n+1} - \Entro^n 
		\leq \int_\Omega  \mathfrak{u}_\M^\infty \Phi_1'\left(  \frac{\mathfrak{u}^{n+1}_\M}{\mathfrak{u}^\infty_\M}\right ) 
			\frac{\mathfrak{u}^{n+1}_\M - \mathfrak{u}^{n}_\M }{\mathfrak{u}^\infty_\M}.
\]
Then, since $w_\M^{n+1} = \log \left ( \frac{\mathfrak{u}_\M^{n+1}}{\mathfrak{u}_\M^\infty} \right)$ and $\Phi_1 ' = \log$, one has
\begin{equation} \label{C4:eq:nl:conv:entro}
	\Entro^{n+1} - \Entro^n 
		\leq  \int_\Omega w_\M^{n+1} \left (\mathfrak{u}^{n+1}_\M - \mathfrak{u}^{n}_\M \right ).
\end{equation}	
On the other hand, testing~\eqref{C4:sch:test} with $\w_\D^{n+1} \in \V_\D^{k}$, and using~\eqref{eq:TD}, we get 
\[
	\int_\Omega w_\M^{n+1} \left (\mathfrak{u}^{n+1}_\M - \mathfrak{u}^{n}_\M \right )
		= - \Delta t \,\mathcal{T}_\D(\l_\D^{n+1}; \l_\D^{n+1} +  \phid_\D, \w^{n+1}_\D)		
		= - \Delta t \,\mathcal{T}_\D(\l_\D^{n+1}; \w_\D^{n+1}, \w^{n+1}_\D),
\]
which finally yields~\eqref{C4:eq:nl:entrorel} by definition of the discrete dissipation.
\end{proof}
\noindent
Remark that since $\mathfrak{u}_\M^0=u^0$, and $u^0\geq 0$ in $\Omega$, $u^0\in L^1(\Omega)$, and $\int_\Omega u^{0} \log( u^{0} ) < \infty$, one has $\Entro^0 < \infty$.
Note finally that the previous results hold true for any $\varepsilon \geq 0$ in~\eqref{C4:sch:stationnary}.


In the rest of this section, we focus on the existence of solutions and on their long-time behaviour. We henceforth assume that $\varepsilon > 0$.
The proofs for both results rely on a discrete a priori estimate, which is obtained by means of a high-order counterpart of~\cite[Lemma 2]{CHHLM:22}.
In order to perform the analysis, we first introduce an inner product $\langle \cdot , \cdot \rangle$ on $\V_\D^{k}$: 
\[
	\forall \z_\D, \v_\D \in \V_\D^{k}, \qquad 
	\left \langle \z_\D, \v_\D \right \rangle = 
	\sum_{K\in\M}\bigg(\int_K z_K v_K 
	+ \sum_{\s \in \E_K} h_\s \int_\s (z_K-z_\s) (v_K-v_\s)\bigg).
\]
We denote by $\|{\cdot}\|$ the corresponding Euclidean norm: 
\[
	\forall \v_\D \in \V_\D^{k}, \qquad 
	\| \v_\D \|^2 = \sum_{K \in \M } \bigg(\|v_K\|_{L^2(K)}^2  + \sum_{\s \in \E_K } h_\s\|v_K-v_\s\|_{L^2(\s)}^2\bigg).
\]
\begin{lemma}[Discrete boundedness by mass and energy semi-norm] \label{C4:lem:BbE}
Let $\l_\D \in \V_\D^{k}$, and assume that there exist $C_\sharp >0$ and $M_\sharp \geq M_\flat > 0$ such that 
\begin{equation} \label{C4:eq:bounds:nl}
	M_\flat \leq \int_\Omega \e^{\ell_\M} \leq M_\sharp 
	\qquad \text{ and } \qquad 
	|\l_\D |_{1,\D} \leq C_\sharp.
\end{equation}
Then, there exists a positive constant $C$, only depending on $M_\flat$, $M_\sharp$, $C_\sharp$, $\Omega$, $\theta$, $k$ and $h_\flat$ such that 
\[
	\| \l_\D \| \leq C.
\]
\end{lemma}
\begin{proof}
  Let us first remark that
  \[
  \sum_{K\in\M}\sum_{\s\in\E_K}h_\s\|\ell_K-\ell_\s\|_{L^2(\s)}^2\leq h_\D^{2}\sum_{K\in\M}\sum_{\s\in\E_K}\frac{1}{h_\s}\|\ell_K-\ell_\s\|_{L^2(\s)}^2\leq h_\D^{2}|\l_\D|_{1,\D}^2\leq {\rm diam}(\Omega)^2 C_\sharp^2.
  \]
  Hence, to estimate $\| \l_\D\|$, all \rev{that} remains to bound is $\|\ell_\M\|_{L^2(\Omega)}$.
  Recalling the notation $\langle z\rangle=\frac{1}{|\Omega|_d}\int_{\Omega}z$, and applying the discrete Poincar\'e--Wirtinger inequality~\eqref{C4:eq:poinca}, it holds
  \begin{equation} \label{eq:PW}
  \|\ell_\M-\langle\ell_\M\rangle\|_{L^2(\Omega)}\leq C_{PW}|\l_\D|_{1,\D}\leq C_{PW}C_\sharp,
  \end{equation}
  with $C_{PW}>0$ only depending on $\Omega$, $\theta$ and $k$.
  Thus, by the triangle inequality, we infer
  \[
  \|\ell_\M\|_{L^2(\Omega)}\leq\|\ell_\M-\langle\ell_\M\rangle\|_{L^2(\Omega)}+\|\langle\ell_\M\rangle\|_{L^2(\Omega)}\leq C_{PW}C_\sharp+|\Omega|_d^{\nicefrac12}|\langle\ell_\M\rangle|,
  \]
  and we are only left with estimating $|\langle\ell_\M\rangle|$. We proceed in two steps, showing first an upper bound on $\langle\ell_\M\rangle$, and then a lower bound. Applying Jensen's inequality, the upper bound can be readily obtained:
  \[
  \e^{\langle\ell_\M\rangle}\leq \langle\e^{\ell_\M}\rangle\leq\frac{M_\sharp}{|\Omega|_d},
  \]
  which yields $\langle\ell_\M\rangle\leq\log\big(\frac{M_\sharp}{|\Omega|_d}\big)$. To prove the lower bound, we start from~\eqref{eq:PW}, and we use local reverse Lebesgue embedding (cf.~\cite[Lemmas 1.25 and 1.12]{DPDro:20}). This yields
  \[
  \|\ell_\M-\langle\ell_\M\rangle\|_{L^\infty(\Omega)}\leq C_{RL}h_\flat^{-\nicefrac{d}{2}}\|\ell_\M-\langle\ell_\M\rangle\|_{L^2(\Omega)}\leq C_{PW}C_\sharp C_{RL}h_\flat^{-\nicefrac{d}{2}},
  \]
  where $C_{RL}>0$ only depends on $d$, $\theta$ and $k$. Then, remarking that
  \[
  \e^{\ell_\M}=\e^{\langle\ell_\M\rangle}\e^{\left(\ell_\M-\langle\ell_\M\rangle\right)}\leq\e^{\langle\ell_\M\rangle}\e^{C_{PW}C_\sharp C_{RL}h_\flat^{-\nicefrac{d}{2}}},
  \]
  and integrating over $\Omega$, we get
  \[
  \int_{\Omega}\e^{\ell_\M}\leq\e^{\langle\ell_\M\rangle}|\Omega|_d\e^{C_{PW}C_\sharp C_{RL}h_\flat^{-\nicefrac{d}{2}}}.
  \]
  Now, using the lower bound on $\int_{\Omega}\e^{\ell_\M}$, and taking the logarithm, we finally infer that
  \[
  \log\left(\frac{M_\flat}{|\Omega|_d\e^{C_{PW}C_\sharp C_{RL}h_\flat^{-\nicefrac{d}{2}}}}\right)\leq\langle\ell_\M\rangle.
  \]
  This concludes the proof.

\end{proof}

\noindent
We now state the existence result, which holds true for positive $\varepsilon$. 
The proof adopts the methodology developed in~\cite{CHHLM:22} in the (nonlinear) HFV context.

\begin{restatable}[Existence of solutions to the nonlinear scheme~\eqref{C4:sch:AD}]{theorem}{propnlExist} \label{C4:th:nl:exist}
Assume that the stabilisation parameter $\varepsilon$ in~\eqref{C4:sch:stationnary} is positive.
Then, there exists at least one solution $\left (\l_\D^n \right)_{n \geq 1}$ to the scheme \eqref{C4:sch:AD}. The corresponding discrete densities $\left (\u_\D^n \right ) _{n \geq 1}$, defined by~\eqref{eq:rec.pos.den}, are positive.
\end{restatable}
\begin{proof}
The proof proceeds by induction. Let $n \in \mathbb{N}$, and assume that $\mathfrak{u}_\M^n$ is well defined, following~\eqref{C4:sch:u} (if $n \geq 1$) or~\eqref{C4:sch:ini} (if $n=0$).
We now prove the existence of a solution $\l_\D^{n+1} \in \V_\D^{k}$ to~\eqref{C4:sch:test}. 
For convenience, instead of looking for the discrete logarithm potential, we will equivalently seek for the discrete quasi-Fermi potential $\w_\D^{n+1} = \l_\D^{n+1} + \phid_\D - \log(c^M_{nl}) \one_\D$ (cf.~\eqref{eq:qua.fer}).  

First, notice that, given any $\w_\D \in \V_\D^{k}$, and corresponding discrete logarithm potential $\l_\D$ (through~\eqref{eq:qua.fer}) and discrete density $\u_\D $ (through~\eqref{eq:rec.pos.den}), the map
\[
	\v_\D \mapsto \int_\Omega \frac{\mathfrak{u}_\M - \mathfrak{u}_\M^n}{\Delta t } v_\M 
		+ \mathcal{T}_\D(\l_\D; \w_\D, \v_\D)
\] 
is a bounded linear form on $\V_\D^{k}$.
Therefore, by the Riesz--Fr\'echet representation theorem, there exists a unique element $\G(\w_\D) \in \V_\D^{k}$ such that 
\[
	\forall \v_\D \in \V_\D^{k}, \qquad 
		\left \langle  \G(\w_\D) , \v_\D  \right \rangle 
		 = \int_\Omega \frac{\mathfrak{u}_\M - \mathfrak{u}_\M^n}{\Delta t } v_\M 
		 + \mathcal{T}_\D(\l_\D; \w_\D, \v_\D).
\] 
Remark that
$\w_\D \mapsto \G(\w_\D)$ is a continuous (nonlinear) map of $\V_\D^{k}$.
Note also that, for any discrete quasi-Fermi potential $\w_\D \in \V_\D^{k}$ such that $\G(\w_\D) = \zero_\D$, by~\eqref{eq:TD}, the corresponding discrete logarithm potential $\l_\D$ solves~\eqref{C4:sch:test}.
Our aim from now on is thus to show that $\G$ does vanish on $\V_\D^{k}$.

To this purpose, we introduce a regularisation of $\G$: given any $\mu > 0$, we let 
\[
	\G^\mu :  \V^k_\D\to\V^k_\D;\;\w_\D \mapsto \G(\w_\D) + \mu \w_\D.
\]
By definition of $\G$, one has 
\[
\begin{split}
    \left \langle  \G^\mu(\w_\D) , \w_\D  \right \rangle 
    			& =  \left \langle  \G(\w_\D) , \w_\D  \right \rangle  + \mu \| \w_\D\|^2 \\
    		& = \int_\Omega \frac{\mathfrak{u}_\M - \mathfrak{u}_\M^n}{\Delta t } w_\M  + \mathcal{T}_\D(\l_\D; \w_\D, \w_\D) + \mu\| \w_\D\|^2.
  \end{split}
\]
As already shown in the proof of Proposition~\ref{C4:prop:nl:apriori} (cf.~\eqref{C4:eq:nl:conv:entro}), by convexity of $\Phi_1$, one has 
\[
	\int_\Omega \frac{\mathfrak{u}_\M - \mathfrak{u}_\M^n}{\Delta t } w_\M  \geq \frac{\Entro(\w_\D) - \Entro^n}{\Delta t},
\]
where the discrete entropies are defined by
\[
	\Entro(\w_\D) = \int_\Omega \mathfrak{u}_\M^\infty \Phi_1 \left ( \frac{\mathfrak{u}_\M}{\mathfrak{u}_\M^\infty} \right ) 
	\quad \text{ and } \quad 
	\Entro^n =  \int_\Omega \mathfrak{u}_\M^\infty \Phi_1\left ( \frac{\mathfrak{u}_\M^n}{\mathfrak{u}_\M^\infty} \right ).
\]
As already mentioned, since $\Phi_1$ is a non-negative function, these two quantities are non-negative.
Note that it may occur that $\Entro^n = 0$ (which is equivalent to $\mathfrak{u}^n_\M = \mathfrak{u}_\M^\infty$ in $\Omega$ for $n\geq 1$, or $u^0=\mathfrak{u}_\M^\infty$ in $\Omega$), in which case $\l_\D = \l_\D^\infty$ is the unique solution to~\eqref{C4:sch:test} (uniqueness follows from the entropy relation~\eqref{C4:eq:nl:entrorel}).
In the following, we therefore assume that $\Entro^n > 0$.
The previous identities, and the non-negativity of the dissipation and entropy, imply that 
\begin{equation} \label{C4:eq:Gww} 
\begin{split}
    \left \langle  \G^\mu(\w_\D) , \w_\D  \right \rangle 
    			& \geq  \frac{\Entro(\w_\D) - \Entro^n}{\Delta t}  
    					+ \mathcal{T}_\D(\l_\D; \w_\D, \w_\D) + \mu \| \w_\D\|^2 \\ 
    			& \geq   \mu\| \w_\D\|^2 -\frac{ \Entro^n }{\Delta t } .
  \end{split}
\end{equation}
Letting $r=\sqrt{  \frac{\Entro^n}{\mu \Delta t }}>0$, one has that $\left \langle  \G^\mu(\w_\D) , \w_\D  \right \rangle \geq 0 $ for all $\w_\D \in \V_\D^{k}$ such that $\| \w_\D \| =r$.
Therefore, according to~\cite[Lemma 1]{CHHLM:22} (cf.~also \cite[Section 9.1]{Evans:10}), which is a by-product of Brouwer's fixed-point theorem, there exists $\w_\D^\mu \in  \V_\D^{k}$ such that 
\begin{equation} \label{C4:eq:wmu}
	\G^\mu(\w_\D^\mu) = \zero_\D 
	\qquad \text{ and } \qquad 
	\| \w_\D^\mu \| \leq \sqrt{  \frac{\Entro^n}{\mu \Delta t }}.
\end{equation}

Now, plugging $\w_\D^\mu$ in~\eqref{C4:eq:Gww}, and using that $\G^\mu(\w_\D^\mu) = \zero_\D$, we get 
\[
	\frac{\Entro(\w_\D^\mu)}{\Delta t} +   \mathcal{T}_\D(\l_\D^\mu; \w_\D^\mu, \w_\D^\mu) + \mu \| \w_\D^\mu\|^2  \leq \frac{ \Entro^n }{\Delta t }, 
\] 
so that $\mathcal{T}_\D(\l_\D^\mu; \w_\D^\mu, \w_\D^\mu)  \leq \frac{ \Entro^n }{\Delta t }$. 
Thus, recalling the definition~\eqref{C4:sch:stationnary}--\eqref{C4:sch:stationnary:global} of $\mathcal{T}_\D$, as well as the stability estimate~\eqref{C4:eq:coercivity} for $a_\D$, we infer that
\[
	\varepsilon\, h_\flat^{k+2} | \w_\D^\mu|_{1,\D}^2 \lesssim \frac{ \Entro^n }{\Delta t }. 
\] 
On the one hand, by~\eqref{eq:qua.fer} and the estimate~\eqref{C4:eq:bounds:phi} on $|\phid_\D|_{1,\D}$, it holds
\begin{equation} \label{C4:eq:bound:lmu}
	|\l_\D^\mu |_{1,\D} = |\w_\D^\mu - \phid_\D  |_{1,\D} 
		\lesssim \sqrt{\frac{ \Entro^n }{\varepsilon h_\flat^{k+2} \Delta t }} +1.
\end{equation}
On the other hand, by definition of $\G^\mu$, one first infers that
\[
\begin{split}
    0 = \left \langle  \G^\mu(\w_\D^\mu) , \one_\D  \right \rangle 
    			& =  \int_\Omega \frac{\mathfrak{u}_\M^\mu - \mathfrak{u}_\M^n}{\Delta t } 
    				+ \mathcal{T}_\D(\l_\D^\mu; \w_\D^\mu, \one_\D ) + \mu \left \langle \w_\D^\mu ,  \one_\D \right \rangle   \\  		
    			& = \int_\Omega \frac{\mathfrak{u}_\M^\mu - \mathfrak{u}_\M^n}{\Delta t }  +  \mu \left \langle \w_\D^\mu ,  \one_\D \right \rangle.
\end{split}
\]
Second, using the Cauchy--Schwarz inequality, followed by the bound~\eqref{C4:eq:wmu} on $\|\w_\D^\mu\|$, one gets
\[
	\left |\int_\Omega \left (\mathfrak{u}_\M^\mu - \mathfrak{u}_\M^n \right )\right | 
			\leq \mu \Delta t \| \w_\D^\mu \|  \|  \one_\D  \|
			\leq \sqrt{\mu}   \sqrt{  \Delta t \,\Entro^n \,|\Omega|_d},
                        \]
  where we have also used that $\|\one_\D\|=|\Omega|_d^{\nicefrac12}$.                      
Thus, letting $M^{n} =\int_\Omega \mathfrak{u}_\M^n > 0$ (recall that $\int_{\Omega}\mathfrak{u}^0_\M=M>0$), and $\mu^n = \frac{(M^{n})^2}{4  \Delta t \Entro^n |\Omega|_d} > 0$, for all $0 < \mu \leq   \mu^n$ one has
\begin{equation} \label{C4:eq:bound:massmu}
	\frac{M^n}{2}\leq \int_\Omega \e^{\ell_\M^\mu} \leq  \frac{3M^n}{2}.
\end{equation}
Leveraging~\eqref{C4:eq:bound:lmu} and~\eqref{C4:eq:bound:massmu}, one can eventually apply Lemma~\ref{C4:lem:BbE} with $M_\flat = \frac{M^n}{2}$, $M_\sharp = \frac{3M^n}{2}$, and $C_\sharp$ proportional to $\sqrt{\frac{ \Entro^n }{\varepsilon h_\flat^{k+2} \Delta t }} +1$ (note that these three constants do not depend on $\mu$): there exists a constant $C>0$, independent of $\mu$, such that 
\[
	\forall \mu \in (0, \mu^n], \qquad 
		\| \l_\D^\mu\|\leq C.
\] 
Then, by compactness, there exists $\l_\D^{n+1} \in \V_\D^{k}$ such that, up to extraction (not relabelled), $\l_\D^\mu \to \l_\D^{n+1}$ when $\mu \to 0$. On the other hand, $\G^\mu$ tends to $\G$ as $\mu$ tends to $0$. Therefore, letting $\w_\D^{n+1} = \l_\D^{n+1} + \phid_\D - \log(c^M_{nl}) \one_\D$, we have 
$ \zero_\D = \G^\mu(\w_\D^\mu) \to \G(\w_\D^{n+1})$ as  $\mu \to 0$, which implies that 
\[
	\G(\w_\D^{n+1}) = \zero_\D. 
\] 
It follows that $\l_\D^{n+1}$ is a solution to~\eqref{C4:sch:test}.
\end{proof}

\begin{rem}[Uniqueness of the solution]
  As for the low-order nonlinear VAG, DDFV and HFV schemes of~\cite{CaGui:17,CCHKr:18,CHHLM:22}, the uniqueness of the solution to~\eqref{C4:sch:AD} is still an open question. 
  A possible approach to show such a result could be to consider the relative discrete entropy of a solution with respect to another solution, and show that this quantity vanishes. 
\end{rem}

Last, we study the long-time behaviour of the nonlinear HHO scheme.

\begin{restatable}[Long-time behaviour of the nonlinear scheme]{prop}{propnlLT} \label{C4:prop:nl:LT}
Assume that the stabilisation parameter $\varepsilon$ in~\eqref{C4:sch:stationnary} is positive, and let $\left (\l_\D^n \right)_{n \geq 1}$ be a solution to the scheme \eqref{C4:sch:AD}. 
Then, the discrete solution converges in time towards the discrete equilibrium logarithm potential:  
\begin{equation} \label{C4:eq:nl:CV:LT}
	\l_\D^n \xrightarrow[n\to\infty]{}\l_\D^\infty \;\text{ in } \V_\D^{k}.
\end{equation}
Consequently, the corresponding \rev{reconstructed} discrete density $(\mathfrak{u}^n_\M)_{n \geq 1}$ converges to $\mathfrak{u}_\M^\infty$ in $L^\infty(\Omega)$.
\end{restatable}
\begin{proof}
First, remark that owing to the entropy relation~\eqref{C4:eq:nl:entrorel}, one has
\[
	\sum_{n \geq 1 } \Diss^n \leq \sum_{n \geq 1 } \frac{\Entro^{n-1} - \Entro^{n}}{\Delta t} \leq \frac{\Entro^0}{\Delta t}.
\]
Thus, according to the definition of the discrete dissipation $\Diss^n$, alongside with the definition~\eqref{C4:sch:stationnary}--\eqref{C4:sch:stationnary:global} of $\mathcal{T}_\D$, and the stability estimate~\eqref{C4:eq:coercivity} for $a_\D$, we infer that 
\[
	\sum_{n \geq 1}  | \w_\D^n|_{1,\D}^2 \lesssim  \frac{\Entro^0}{ \varepsilon h_\flat^{k+2}  \Delta t}.
\]
This implies, in particular, that 
\begin{equation} \label{C4:eq:diss:sommable}
	\forall n \geq 1, \quad
	| \w_\D^n|_{1,\D} \lesssim \sqrt{\frac{\Entro^0}{ \varepsilon h_\flat^{k+2} \Delta t}}
	\qquad \text{ and } \qquad 
	| \w_\D^n|_{1,\D} \xrightarrow[n\to\infty]{} 0.
\end{equation}
Let $n \geq 1$. By~\eqref{eq:qua.fer} and~\eqref{C4:eq:bounds:phi}, one has 
$ |\l_\D^n|_{1,\D} \lesssim \sqrt{\frac{\Entro^0}{ \varepsilon h_\flat^{k+2} \Delta t} }+1
$. 
On the other hand, by the mass preservation~\eqref{C4:eq:nl:mass}, we have $\int_\Omega \e^{\ell^n_\M} = M  > 0$.
Therefore, one can apply Lemma~\ref{C4:lem:BbE}, and infer the existence of a positive constant $C$ (which is independent of $n$) such that 
\begin{equation} \label{eq:bound}
	\forall n \geq 1, \qquad \| \l_\D^n \| \leq C.
\end{equation}
It follows, by compactness, that there exists $\l_\D \in \V_\D^{k}$ such that, up to extraction (not relabelled), 
\[
	\lim_{n \to \infty} \l_\D^n  = \l_\D  \;\text{ in } \V_\D^{k}.
\]
By~\eqref{C4:eq:diss:sommable},~\eqref{eq:qua.fer}, and continuity of $|{\cdot}|_{1,\D}$ on $\V_\D^{k}$, we infer that 
\[
	|\l_\D + \phid_\D |_{1,\D} = 0.
\]
This means that there exists $a \in \R$ such that $\l_\D + \phid_\D = a \one_\D$. By mass preservation, we get 
\[
	M = \int_\Omega \e^{\ell_\M^n} \xrightarrow[n\to\infty]{} \int_\Omega \e^{\ell_\M},
\]
so that $a = \log(c^M_{nl}) $, which implies that $\l_\D =\log(c^M_{nl})\one_\D -  \phid_\D = \l_\D^\infty$. 
By uniqueness of the limit, we finally infer the convergence of the whole sequence $\left( \l_\D^n \right )_{n \geq 1}$ towards $\l_\D^\infty$ in $\V^k_\D$.
This implies, in particular, that $\ell^n_\M \xrightarrow[n\to\infty]{} \ell_\M^\infty$ in $L^\infty(\Omega)$, by norm equivalence in finite-dimensional vector spaces.
Then, by the mean value theorem, we deduce that 
\[
	\|\mathfrak{u}_\M^n - \mathfrak{u}_\M^\infty \|_{L^\infty(\Omega)} = \|\e^{\ell_\M^n} - \e^{\ell_\M^\infty} \|_{L^\infty(\Omega)} 
		\leq \e^{ \max( \|\ell_\M^n \|_{L^\infty(\Omega)} , \|\ell_\M^\infty \|_{L^\infty(\Omega)}  )}
		\|\ell_\M^n - \ell_\M^\infty \|_{L^\infty(\Omega)}, 
\]
which implies, by uniform boundedness (in $n$) of $\left( \ell_\M^n \right )_{n \geq 1}$, the convergence of the \rev{reconstructed} discrete density in $L^\infty(\Omega)$.
\end{proof}

\begin{rem}[Non-uniformity of the bounds]
  Note that the estimate~\eqref{eq:bound} on the solution to~\eqref{C4:sch:AD} is not uniform with respect to the discretisation parameters $h_\D$ and $\Delta t$, nor with respect to the stabilisation parameter $\varepsilon$. Indeed, having a closer look to the dependencies of the corresponding upper bound (using Lemma~\ref{C4:lem:BbE}), one realises that it blows up as soon as either $h_\D$, $\Delta t$ or $\varepsilon$ tends to zero.
\end{rem}

\begin{rem}[Convergence to equilibrium] \label{C4:rem:LT:speed}
  Notice that the time-asymptotic result of Proposition~\ref{C4:prop:nl:LT} is relatively weaker than the one of Proposition~\ref{C4:prop:ExpF:time} in the exponential fitting context. For the latter result, the convergence to equilibrium is shown to be exponentially fast, and the decay rate $\nu_\omega$ uniform with respect to the discretisation parameters.
  The numerical results of~\cite{Moatt:23HO} and Section~\ref{C4:sec:LT} indicate that, also for the nonlinear scheme, the convergence is expected to be exponential, with seemingly uniform (and close to the PDE model one) decay rate.
  At the theoretical level, to prove exponential convergence to equilibrium, one has to establish a control of the discrete entropy by the discrete dissipation. 
  Adapting the arguments from~\cite[Theorem 3]{Moatt:23SC}, such a control can actually be established in the present context, but leads to a non-uniform decay rate (also depending on $\varepsilon$), and to a final result still only valid in the case $\varepsilon>0$.
  In order to showcase a uniform (and $\varepsilon$-independent) decay rate, and establish a result also valid in the case $\varepsilon=0$, a (high-order) uniform discrete Logarithmic-Sobolev inequality needs to be available (cf.~\cite{CHHLM:22} in the low-order HFV context).
  This is the subject of ongoing research. 
\end{rem}

\section{Numerical results} \label{C4:sec:num}

In this section, we extensively assess the high-order nonlinear scheme~\eqref{C4:sch:AD}. We study positivity preservation, convergence, efficiency (accuracy vs.~computational cost), and long-time behaviour. We also compare it, in terms of structure preservation, with the linear high-order exponential fitting scheme~\eqref{C4:sch:ExpF}.
All the test-cases considered below are set in the two-dimensional domain $\Omega = (0,1)^2$, and are (except for the last one) taken from~\cite{CHHLM:22}, to which we refer for more detailed descriptions.
Given a (face) degree $k\geq 0$, the nonlinear scheme~\eqref{C4:sch:AD} will be referred to as \texttt{nlhho\_k}, whereas the exponential fitting one~\eqref{C4:sch:ExpF} as \texttt{expf\_k}.
For the nonlinear scheme, we will always use below the value $\varepsilon=1$ for the parameter $\varepsilon$ in~\eqref{C4:sch:stationnary}.
However, in some situations, we will compare the two values $\varepsilon = 1$ and $\varepsilon =0$. The nonlinear scheme with $\varepsilon = 0$ will then be denoted \texttt{nlhho\_k\_0}.

\subsection{Implementation}


All numerical tests presented below have been run on a laptop equipped with an Intel Core i7-9850H processor clocked at 2.60GHz and 32Gb of RAM.
Our HHO implementation makes use of monomial basis functions for both the cell and face unknowns. 
Such a choice is known to introduce numerical instabilities for large values of $k$, we thus restrict our study to $k\leq 3$.
The use of orthonormal basis functions, which is expected to improve on this situation (in particular for the convergence of the Newton algorithm in the nonlinear case), shall be studied in future works.
We use quadrature formulas based on the Dunavant rules~\cite{Dunav:85} (after subtessellation).
To cope with non-polynomial integrands, we employ quadrature formulas of order $2k+5$.
We performed a few tests (not reported here) with higher-order formulas, and did not observe any significant changes.
Last, the local computations are performed sequentially. One could expect a significant gain in terms of performances parallelising the latter.
We discuss below some important implementation aspects for both schemes.

\subsubsection{Exponential fitting scheme}

For the linear exponential fitting scheme, the implementation follows the classical HHO strategy for linear diffusion problems.
We directly solve for the discrete Slotboom variable $(\ro_\D^n)_{n \geq 1}$.
As standard for skeletal methods, we do not solve the full linear system, but first perform static condensation, which allows one to locally eliminate the cell unknowns.
Since the scheme relies on the same LHS matrix at each time step, we perform once and for all an LU decomposition of the matrix at the beginning of the computation.
At each time step, the solution is then inexpensive (the RHS has to be updated, but only through a matrix-vector product).
 
We do not address in this work the main questions which were highlighted in~\cite[Section 5.1.2]{CHHLM:22} in the low-order HFV context, about the (harmonic) averaging of $\omega$ (which is related to the choice of quadrature formulas for the high-order scheme), and the preconditioning of the system (which was equivalent, in the simple HFV context, to choose to solve the system in the density variable).
These aspects shall be investigated in future works. 
Nonetheless, in view of the results obtained in~\cite{CHHLM:22} for the HFV exponential fitting scheme, we believe these potential improvements will have no effect on the positivity violation issues.

\subsubsection{Nonlinear scheme} \label{C4:sec:num:tricks}

The numerical scheme~\eqref{C4:sch:AD} requires to solve a nonlinear system of equations at each time step.
For $n \in \mathbb{N}$, one wants to find $\l_\D^{n+1} \in \V_\D^{k}$ solution to~\eqref{C4:sch:AD}: this scheme can be written as the equation 
\[
	\G^{n,\Delta t} (\l_\D^{n+1}) = \zero_\D, 
\]
with $\G^{n,\Delta t} : \V_\D^{k} \to \V_\D^{k}$ smooth (nonlinear) vector field.
Numerically, to find a zero of $\G^{n,\Delta t}$, we use a Newton method. 

In practice, the use of a naive method without any adaptation proves not to be enough to compute a solution in general. 
In order to get a robust implementation, which can handle various data and meshes, one has to deploy a few techniques. 
For further use, we let $\|\l_\D\|_{l^\infty} $ denote the $l^\infty$-norm of the coefficients of $\l_\D$ in the (cell and face) polynomial bases.
The map $\|{\cdot}\|_{l^\infty} $ is a norm on $\V_\D^{k}$, which is easily (and at very low cost) computable in practice.
To fix the ideas and the notation, the Newton method is defined as follows: given an initialisation $\l_{\D,0} \in  \V_\D^{k}$, and a time step $ 0 < \delta t \leq \Delta t$, one defines a sequence $(\l_{\D,i})_{i\geq 0}$ of elements of $\V_\D^{k}$ such that 
\begin{equation} \label{C4:eq:Newton}
	J^{n,\delta t}_{\l_{\D,i}} (\l_{\D,i+1} - \l_{\D,i} ) = - \G^{n,\delta t} \left ( \l_{\D,i} \right ), 
\end{equation}
where $\G^{n,\delta t}$ is the vector field associated to the nonlinear scheme~\eqref{C4:sch:AD} with time step $\delta t$ instead of $\Delta t$, and $J^{n,\delta t}_{\l_{\D,i}}$ is the differential (Jacobian in practice) of $\G^{n,\delta t}$ at $\l_{\D,i}$.
Note that, in practice, we do not solve this linear system, but perform static condensation in order to (locally) eliminate the cell unknowns. 
The resulting linear system is called "condensed system" in what follows.
We discuss below the main tricks deployed to reach robustness in the implementation of the Newton algorithm.
\begin{itemize}	
	\item[\textbf{Stopping criterion.}] 
We define the relative norm of the residual $r_{i+1}$, and the norm of the objective function $g_i$, as 
\[
	r_{i+1} = \frac{\| \l_{\D,i+1} - \l_{\D,i} \|_{l^\infty}}{ \| \l_{\D,i} \|_{l^\infty}}
	\quad \text{ and } \quad 
	g_i = \| \G^{n,\delta t} ( \l_{\D,i}  )  \|_{l^\infty}.
        \]
We consider that the Newton method has converged when either
\[
	\left ( r_{i+1}\leq 0.1 \times \texttt{tol}  \right ) 
	 \quad\text{ or }\quad
	 \left (  r_{i+1} \leq  \texttt{tol}  \text{ and } g_i \leq \texttt{tol}  \right),
\]
with $\texttt{tol} = 5 . 10^{-10}$, in which case we set $\l_\D^{n+1} = \l_{\D,i+1}$.
On the other hand, if this criterion is not met at $i=50$, the method is considered as non-convergent (and we then proceed with a time step reduction, see below). 
In practice, for the tests collected in this article, we never reached $i = 50$, either because the method converged, or because of a loop break (see below).
		
	\item[\textbf{Loop break for unreasonably large $\l_\D$.}]
The computations of $\G^{n,\delta t} \left ( \l_{\D,i} \right )$ and $J^{n,\delta t}_{\l_{\D,i}}$ imply punctual evaluations of $\e^{\ell_{K,i}}$ (for $K\in\M$) and $\e^{\ell_{\s,i}}$ (for $\s\in\E$) in the quadrature formulas.
Such computations can lead to severe numerical issues if the values at the quadrature nodes are too large.
Therefore, we declare that $\l_{\D,i}$ is unreasonably large for the computations if there exists a cell quadrature node $x_{K,q} \in \overline{K}$, or a face quadrature node $x_{\s,q} \in \overline{\s}$, such that 
\[
	|\ell_{K,i}(x_{K,q})| \geq 100 
	\quad \text{ or } \quad 
	|\ell_{\s,i}(x_{\s,q})| \geq 100.
\]
In such a case, the method is immediately considered as non-convergent, and we proceed with a time step reduction (see below).
Note that the choice of the value $100$ allows one to compute densities $\u_\D$ over a range from $10^{-43}$ to $10^{43}$, and hence should not be a significant restriction in practice. 
In the numerical simulations presented below, the use of this loop-breaking procedure is absolutely necessary in order to avoid the evaluation of too large quantities, leading to some ``explosion'' of the method and crash of the code.
Moreover, we also operate a loop break if the linear solver does not perform a successful resolution of the condensed linear system associated to~\eqref{C4:eq:Newton}, which corresponds to situations for which either $J^{n,\delta t}_{\l_{\D,i}}$ or its condensed counterpart are not invertible. Such situations occur in practice, essentially on very coarse meshes.
	
	\item[\textbf{Adaptive time stepping.}] 
The previous strategies can lead to a solution failure for some given time step $\delta t$.
If the Newton method did not converge, we try to compute the solution for a smaller time step $\delta t/2$. 
On the other hand, if the method did converge, we use for the subsequent time step the larger value $2\delta t$. 
The maximal time step allowed is the initial one, denoted by $\Delta t$.
In practice, the scheme may perform numerous time step reductions at the beginning (early times) of the computation.
	
	\item[\textbf{Initialisation by truncation and filtration.}] 
As for any Newton method, the question of the initialisation is fundamental in order to get a robust implementation. It appears that, for $n \geq 1$, the natural initialisation $\l_{\D, 0} = \l_\D^n$ is satisfactory when used with the adaptative time stepping strategy. 
However, for $n=0$, such a choice is not possible, since $\l_\D^0$ does not exist in general if $u^{0}$ vanishes locally or is too small (cf.~Remark~\ref{C4:rem:nl:ini}). 
A first way of tackling this problem is to define a truncated initial logarithm potential $\tilde{\ell}^{0}$ as 
\[
	\tilde{\ell}^{0} = \log \left ( \max(u^{0}, 10^{-8} )  \right ), 
\]
and to initialise the Newton method with $\tilde{\l}_\D^{0}=\underline{I}^k_\D(\tilde{\ell}^0) \in \V_\D^{k}$, provided one can give a sense to the face components.
In fact, such a strategy entails another limitation: $\tilde{\l}_\D^{0}$ exhibits strong oscillations in the regions where the truncation is performed (this is also true when $u^{0}$ is discontinuous over $\Omega$, as in Section~\ref{C4:sec:pos}).
These oscillations usually make the method diverge, even with extremely small time step.
Therefore, we eventually initialise the method with a ``filtered'' (non-oscillating) discrete logarithm potential $\l^{0}_\D \in \V_\D^{k}$, which corresponds to a zero-order polynomial projection of $\tilde{\ell}^0$:
\[
	\ell^0_K = \Pi_K^0(\tilde{\ell}_{\mid K}^{0})\;\forall K \in \M
	\quad\text{ and }\quad
	\ell^0_\s = \Pi_\s^0(\tilde{\ell}_{\mid\s}^{0})\;\forall \s \in \E,
\]
still provided one can give a sense to the face components.
In practice, using $\l_{\D,0} = \l^{0}_\D$ as the first initialisation (when $n=0$) yields convergent Newton methods for all tests presented below. 
The use of filtered initial discrete data seems particularly crucial for high-order schemes ($k \geq 1$). For the lowest-order version of the scheme ($k=0$), the use of $\tilde{\l}_\D^{0}$ as a first initialisation (for $n=0$) often yields convergent Newton methods (up to time step reduction).
\end{itemize}

\noindent
Of course, the chosen values for the stopping criterion and the thresholds are arbitrary and could be modified. 
However, the set of values advocated here makes the scheme robust enough so as to be capable of computing solutions for all the test-cases in this article.

\begin{rem}[Potentials vs.~densities] \label{C4:rem:stopcrit}
  One of the main differences between the present nonlinear scheme and the low-order HFV ones from~\cite{CHHLM:22} and~\cite{Moatt:23SC} lies in the fact that we use here the potential $\ell$ as our (piecewise polynomial) unknown, whereas the density $u$ was used in the low-order schemes.
  Notice that, in the present context, choosing $u$ as the main variable would require to give a discrete meaning to $\nabla\log(u)$, which is not obvious for the following reason: polynomials of degree $\geq 1$ are not stable by the $\log$ function.
  As a by-product of seeking for a potential, our stopping criterion only provides information on $\ell$, while we are eventually interested in the corresponding density.
  Moreover, our criterion only takes into account the coefficients of the polynomials (through the use of the norm $\|{\cdot}\|_{l^\infty}$), but such a measure does not give much information about the effective behaviour of the unknowns.
  A more relevant stopping criterion could be to consider the residual in terms of densities
\[
	\frac{\| \mathfrak{u}_{\M,i+1} - \mathfrak{u}_{\M,i}\|_{L^2(\Omega)} }{ \|\mathfrak{u}_{\M,i}\|_{L^2(\Omega)}}
\]
(and analogous definition for the face unknowns) in order to ensure a satisfying accuracy on $u$.
  However, the main drawback of such a criterion is its evaluation cost.
  In this work, we thus chose to use instead a purely algebraic stopping criterion on $\ell$, whose cost is marginal. 
  The testing of other stopping criteria will be the subject of future investigations.
  Last, for the HFV schemes of~\cite{CHHLM:22,Moatt:23SC}, the following loop-breaking strategy was used: when the computed density had almost-zero (or even non-positive) components, one performed a time step reduction.
  Here, such a situation cannot occur, since $\ell$ is authorised to take any real value, but this apparent latitude on the potential is in fact pernicious.
  Indeed, situations in which $\ell$ takes \rev{negative values with large magnitude} are actually the counterpart of an almost-zero $u$ for the HFV schemes. Like their counterpart, they lead to divergent Newton methods.
  The main difficulty then lies in the design of a relevant criterion in order to avoid these situations.
\end{rem}

\subsection{Positivity} \label{C4:sec:pos}

This first section is dedicated to assessing discrete positivity preservation. 
For the test considered here, we set the advective potential and the anisotropy tensor to
\[
	\phi(x_1,x_2) = -\left (  (x_1-0.4)^2 +(x_2-0.6)^2   \right )
	\quad \text{ and } \quad 	
	\Lambda = 
	\begin{pmatrix}
		0.8 & 0 \\ 
		0 & 1
	\end{pmatrix} .
\]
For the initial datum, we take
$u^{0} = 10^{-3}  \,\1_{B} +   \1_{\Omega \setminus B}$,
where $B$ is the Euclidean ball
\[
	B=\left \lbrace (x_1,x_2) \in \R^2 \mid (x_1-0.5)^2 + (x_2-0.5)^2 \leq 0.2^2 \right \rbrace .
\]
These data ensure that the solution $u$ is positive on $\R_{+}\times\Omega$. 
We perform the simulation on the time interval $[0, 5 . 10^{-4}]$ with $\Delta t = 10^{-5}$, on a (fine) tilted hexagonal-dominant mesh featuring 4192 cells and 12512 edges.
The computed discrete densities are denoted by $(\u^n_\D)_{1 \leq n \leq N_f}$ and $(\ue^{\omega,n}_\D)_{1 \leq n \leq 50}$. 
Remark that the situation $N_f > 50$ may occur if the nonlinear scheme has to perform time step adaptation.

In Table~\ref{C4:table:positivity}, we collect the minimal values reached by the discrete solutions. 
The values of \texttt{mincellA} (for ``average'') are defined by
\[
	\min \left  \lbrace \frac{1}{|K|_d} \int_K \mathfrak{u}_K^n  \mid K \in \M, 1 \leq n \leq N_f \right \rbrace
	\quad \text{ and } \quad 
		\min \left  \lbrace \frac{1}{|K|_d} \int_K \mathfrak{u}_K^{\omega,n}  \mid K \in \M, 1 \leq n \leq 50 \right \rbrace,
\]
for, respectively, the nonlinear scheme and  the exponential fitting scheme.
The values of \texttt{mincellQN} are the minimal values taken by the densities at the cell quadrature nodes.
Analogous definitions hold for the faces.
The values of \texttt{\#resol} correspond to the number of linear systems solved during the computation.
Note that the size of these systems depends on the value of $k$, so it is not a relevant  information to compare the cost of the schemes for different values of $k$.
Last, \texttt{walltime} is the total time (in $s$) needed to compute the discrete solution (it includes the pre-computation steps, such as the computation of the matrices representing $G_K$).
\begin{table}[h]
\center
{\small
\center
\begin{tabularx}{0.9\textwidth}{|Y||Y|c|c|c|c|c|} 
\hline 
   \texttt{scheme}			& \texttt{walltime} & \texttt{\#resol}	& \texttt{mincellA}  	& \texttt{minfaceA} 	& \texttt{mincellQN} 	&  \texttt{minfaceQN} \\ 
\hline 
\hline 
	\texttt{nlhho\_0} 		& 7.17e+01		& 224 		&  1.00e-03		& 1.01-03 	& 2.41e-06 	&  1.01e-03	\\ 
\hline 	
	\texttt{nlhho\_1} 		& 4.13e+02		& 248		&  6.65e-04		& 2.05e-05 	& 1.78e-04 	&  3.57e-08 	\\ 
\hline 	
	\texttt{nlhho\_2} 		& 1.45e+03		& 251		&  9.50e-04		& 5.99e-04 	& 2.67e-07 	&  1.06e-05	\\ 
\hline 		
	\texttt{nlhho\_3} 		& 3.87e+03		& 254		&  9.85e-04		& 8.58e-04 	& 1.10e-05 	&  1.79e-05	\\ 
\hline	
\hline 	
	\texttt{expf\_0} 			& 5.66e-01 		& 50 		& 1.02e-03 		& 1.89e-03 	& -3.78e-01 	& 1.89e-03 	\\
\hline 	
	\texttt{expf\_1} 			& 2.23e+00 		& 50 		& -1.29e-02 		& -2.40e-01 	& -4.91e-01 	& -3.71e-01 	\\
\hline
	\texttt{expf\_2} 			& 6.34e+00 		& 50 		& -6.14e-03 		& -1.02e-01 	& -5.08e-01 	& -5.35e-01 	\\ 
\hline 
	\texttt{expf\_3} 			& 1.53e+01 		& 50 		& -3.24e-04 		& -1.02e-02 	& -5.52e-01 	& -4.05e-01 	\\ 
\hline
\end{tabularx} 
}
\caption{\textbf{Positivity of discrete solutions.} 
}\label{C4:table:positivity}
\end{table}
Recall that the 
exponential fitting scheme is linear (with corresponding matrix not depending on time), whence its extremely low cost compared to the nonlinear scheme.
Note, however, that when an LU decomposition is unaffordable and an iterative solver has to be used instead, \texttt{nlhho\_k} is approximately ``only'' five times more costly than \texttt{expf\_k}.
The results of Table~\ref{C4:table:positivity} first indicate that, as expected, all nonlinear schemes preserve the positivity of the discrete solution.
On the other hand, none of the linear schemes preserves positivity on the whole domain $\Omega$. In fact, except \texttt{expf\_0}, all linear schemes studied here do not even preserve the average positivity on each cell, in the sense that there exists $K_0 \in \M$ and $n_0$ integer such that 
\[
	\int_{K_0} \mathfrak{u}^{\omega,n_0}_{K_0} < 0.
\]
Moreover, it is interesting to note that the positivity violation peak (which can be approximated by $|\texttt{mincellQN}|$ and $|\texttt{minfaceQN}|$) increases as $k$ increases, whereas in average (values of $|\texttt{mincellA}|$ and $|\texttt{minfaceA}|$) the lack of positivity becomes smaller as the order increases.

At this stage, it is worth pointing out the fact that quantifying the negativity of the solution is much more difficult for high-order schemes, since it is not possible to ``count'' the number of negative values (which correspond to the degrees of freedom for low-order schemes).
While the \texttt{mincellQN} value gives information about the minimum value reached on the whole domain, it does not give any indication about the measure of the set $\lbrace  x\in \overline{\Omega} \mid \mathfrak{u}^{\omega,n}_\M(x)< 0\rbrace$ where the discrete cell unknown takes negative values. 
The same remark applies to \texttt{mincellA}. 
As an attempt to provide an idea of the size of this set, we display in Table~\ref{C4:table:nbnegative} the number of cells with negative average over the whole simulation, defined as the cardinal of the set
\[
	\left \lbrace (K,n) \in \M \times  [\![ 1; 50]\!] \mid \int_K \mathfrak{u}_K^{\omega,n} < 0 \right \rbrace.
\]	
These data reveal that, excluding \texttt{expf\_0} which performs quite well on this particular test, the higher the order, the smaller the size of the negative-average set.  
\begin{table}[h] 
\center
\begin{tabularx}{0.85\textwidth}{|Y||c|c|c|c|c|}
\hline 
\texttt{scheme} & 
\texttt{expf\_0} & \texttt{expf\_1} & \texttt{expf\_2} & \texttt{expf\_3} \\ 
\hline \hline
\texttt{\#cells with negative average} & 
0 & 824 & 136 & 1 \\ 
\hline 
\end{tabularx} 
\caption{\textbf{Number of negative cell averages.}}\label{C4:table:nbnegative}
\end{table}

The previous observations seem to indicate a competition between two phenomena for linear methods. As $k$ increases, the accuracy is improved, and therefore the discrete solution becomes closer to the exact one. 
Hence, in average, high-order schemes compute solutions with smaller area of negativity, and lesser positivity violation.
However, high values of $k$ induce larger oscillations for the polynomial solution: the computed solution takes negative values on smaller sets, but the (pointwise) undershoots become bigger as $k$ increases.
At the end, it seems that there is no hope to get a positive discrete solution on the whole domain $\Omega$ with a linear method. 

%
%
\begin{rem}[An accuracy criterion taking into account positivity]
The previous observations suggest that, for applications in which preserving the positivity of the solution is an essential feature, the accuracy of the scheme should not simply be defined as an $L^p$-distance between the \rev{reconstructed} discrete solution $\mathfrak{u}_\M$ and the exact one $u$.
We believe that a relevant criterion in order to take into account both ``classical accuracy'' (distance between $\mathfrak{u}_\M$ and $u$) and positivity is to look at the relative Boltzmann entropy (or other kinds of relative $\Phi$-entropies as defined in~\cite{BLMVi:14}) with respect to the exact solution, that is 
\begin{equation} \label{C4:eq:error:entro}
	\texttt{Err}(\mathfrak{u}_\M) = \int_\Omega u \,\Phi_1 \left ( \frac{\mathfrak{u}_\M}{u} \right ), 
\end{equation}
where $\Phi_1(s) = s\log(s) -s +1$ for $s> 0$, $\Phi_1(0) = 1$, and $\Phi_1(s)$ takes large values for $s < 0$.
The interest of such a definition is twofold. 
First, the negativity of $\mathfrak{u}_\M$ is penalised.
Second, if $\mathfrak{u}_\M$ is positive and $\int_\Omega \mathfrak{u}_\M  = \int_\Omega u$ (which is the case in practice for problems with homogeneous Neumann boundary conditions), by Csisz\'ar--Kullback inequality (see e.g.~\cite[Lemma 5.6]{CCHHK:20}), one has 
\[
	 \| \mathfrak{u}_\M - u\|_{L^1(\Omega)} \leq  \sqrt{2 \|u\|_{L^1(\Omega)} \texttt{Err}(\mathfrak{u}_\M) }.
\]
\end{rem}

\subsection{Convergence and efficiency} \label{C4:sec:CV}

We here study the convergence as $(h_\D, \Delta t) \to (0,0)$ of the nonlinear scheme for different values of the polynomial degree $k$.
We consider a test-case with known exact solution. 
We set the advective potential and anisotropy tensor to
\[
	 \phi(x_1,x_2) = - x_1 
	\quad  \text{ and } \quad  
	 \Lambda = \begin{pmatrix}
		l_{x_1} & 0 \\ 
		0 & 1
	\end{pmatrix} 
\]
for $l_{x_1}>0$.
The exact solution is then given by 
\begin{equation} \label{C4:eq:testcase}
	u(t,x_1,x_2) = C_1\e^{-\alpha t + \frac{x_1}{2}} \left ( 2\pi \cos( \pi x_1) + \sin(\pi x_1) \right )
	+ 2 C_1 \pi \e^{ x_1 - \frac{1}{2}  }, 
\end{equation}
where $C_1> 0$ and $ \alpha = l_{x_1} \left ( \frac{1}{4} + \pi^2  \right )$. Note that $u^{0}$ vanishes on 
$\{ x_1 = 1 \}$, but for any $t > 0$, $u(t, \cdot) > 0$.
Here, our experiments are performed using $l_{x_1} =  1$ and $C_1 = 10^{-1}$.

%
\begin{figure}[h]
\begin{minipage}[c]{.51\linewidth}
\begin{tikzpicture}[scale= 0.86]
        \begin{loglogaxis}[
            legend style = { 
              at={(0.5,1.1)},
              anchor = south,
              tick label style={font=\footnotesize},
              legend columns=4
            },ylabel=\small{Relative $L^2_t(L^2_x)$-error},xlabel=\small{Mesh size $h_\D$}
          ]
          \addplot[color=blue,mark=triangle*] table[x=Meshsize,y=error_sol] {data/time_CV_adpt/errors_hho_0};          
          \addplot[color=red,mark=square*] table[x=Meshsize,y=error_sol] {data/time_CV_adpt/errors_hho_1}; 
          \addplot[color=OliveGreen,mark=*] table[x=Meshsize,y=error_sol] {data/time_CV_adpt/errors_hho_2};             
          \addplot[color=black,mark=star] table[x=Meshsize,y=error_sol] {data/time_CV_adpt/eps=0/errors_hho_2};
          \addplot[color=brown,mark=diamond*] table[x=Meshsize,y=error_sol] {data/time_CV_adpt/errors_hho_3};
          \logLogSlopeTriangle{0.90}{0.2}{0.1}{2}{black};
          \logLogSlopeTriangle{0.90}{0.2}{0.1}{3}{black};
          \logLogSlopeTriangle{0.90}{0.2}{0.1}{4}{black};
          \logLogSlopeTriangle{0.90}{0.2}{0.1}{5}{black};
          \legend{\tiny \texttt{nlhho\_0} , \tiny \texttt{nlhho\_1} , \tiny \texttt{nlhho\_2} , \tiny \texttt{nlhho\_2\_0}, \tiny \texttt{nlhho\_3} } 
        \end{loglogaxis}
      \end{tikzpicture}    
\end{minipage}
\begin{minipage}[c]{.51\linewidth}
\begin{tikzpicture}[scale= 0.86]
        \begin{loglogaxis}[
            legend style = { 
              at={(0.5,1.1)},
              anchor = south,
              tick label style={font=\footnotesize},
              legend columns=4
            },ylabel=\small{Relative $L^2_t(H^1_x)$-error},xlabel=\small{Mesh size $h_\D$}
          ]          
          \addplot[color=blue,mark=triangle*] table[x=Meshsize,y=error_grad] {data/time_CV_adpt/errors_hho_0};          
          \addplot[color=red,mark=square*] table[x=Meshsize,y=error_grad] {data/time_CV_adpt/errors_hho_1}; 
          \addplot[color=OliveGreen,mark=*] table[x=Meshsize,y=error_grad] {data/time_CV_adpt/errors_hho_2};
          \addplot[color=black,mark=star] table[x=Meshsize,y=error_grad] {data/time_CV_adpt/eps=0/errors_hho_2};
          \addplot[color=brown,mark=diamond*] table[x=Meshsize,y=error_grad] {data/time_CV_adpt/errors_hho_3};
         
          \logLogSlopeTriangle{0.90}{0.2}{0.1}{1}{black};    
          \logLogSlopeTriangle{0.90}{0.2}{0.1}{2}{black};          
          \logLogSlopeTriangle{0.90}{0.2}{0.1}{3}{black};             
          \logLogSlopeTriangle{0.90}{0.2}{0.1}{4}{black};  
          \legend{\tiny \texttt{nlhho\_0} , \tiny \texttt{nlhho\_1} , \tiny \texttt{nlhho\_2} , \tiny \texttt{nlhho\_2\_0}, \tiny \texttt{nlhho\_3} }      
        \end{loglogaxis}
      \end{tikzpicture}
\end{minipage}
\caption{\textbf{Accuracy vs.~mesh size.} Relative errors on triangular meshes.}
\label{C4:fig:CV:evol:order}
\end{figure}
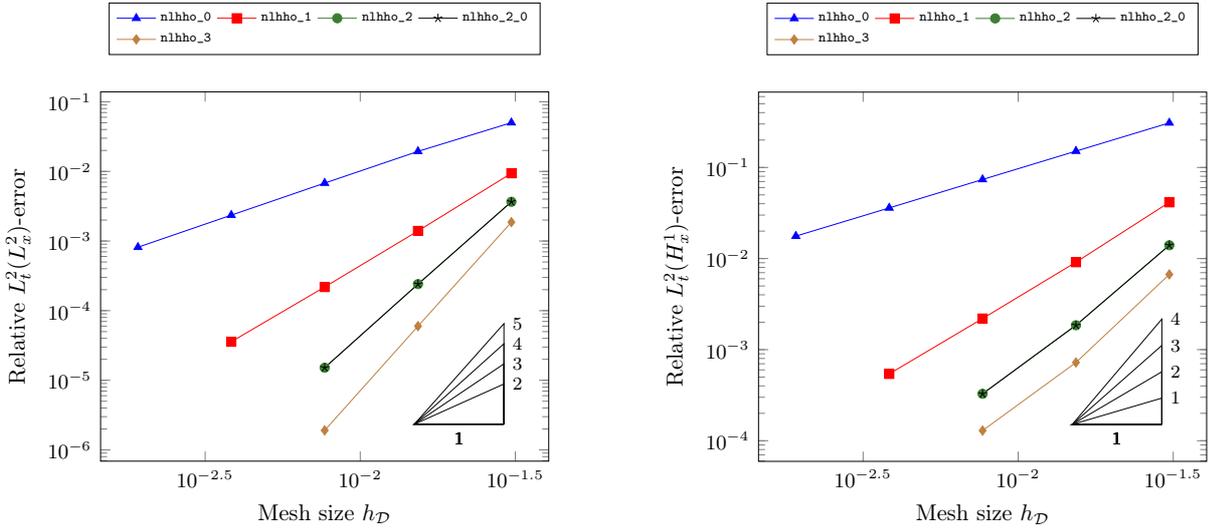

We compute the discrete solutions on the time interval $[0,0.1]$, and we denote by $(\u^n_\D)_{1 \leq n \leq N_f}$ the corresponding discrete densities. We monitor the relative $L^2_t(L^2_x)$-error and $L^2_t(H^1_x)$-error on the solution, respectively defined by 
\[
	 \sqrt{\frac{\sum_{n = 1}^{N_f} \delta t^n  \|\mathfrak{u}_\M^n - u(t^n, \cdot ) \|^2_{L^2(\Omega)} }
	 	{ \sum_{n = 1}^{N_f} \delta t^n  \|u(t^n, \cdot ) \|^2_{L^2(\Omega)}}}
	\qquad\text{ and }\qquad	
	\sqrt{\frac{ \sum_{n = 1}^{N_f}  \delta t^n  \| \gradd_\M(\u^n_\D) - \nabla u(t^n, \cdot ) \|^2_{L^2(\Omega)^d}}
	{\sum_{n = 1}^{N_f} \delta t^n  \|\nabla u(t^n, \cdot ) \|^2_{L^2(\Omega)^d}}}, 
\]
where $\delta t^n = t^{n}-t^{n-1}$, and $\sum_{1 \leq n \leq N_f} \delta t^n  = 0.1$. The discrete gradient of the densities $\gradd_\M(\u^n_\D)$ is defined as follows. For all $K\in\M$, $(\gradd_\M(\u^n_\D))_{\mid K}=\gradd_K(\u^n_K)$, where $\gradd_K(\u^n_K)$ is a smooth vector field on $K$ defined by mimicking at the discrete level the relation $\nabla u = \e^\ell \nabla \ell $:	
\begin{equation} \label{eq:gradd}
	\gradd_K (\u^n_K) = \e^{\ell_K^n}   G_K(\l_K^n) \quad\text{ in }K.
\end{equation}
Notice that, with the chosen error measures, we do not take into account the time $t=0$.
We perform our simulations on a triangular mesh family $(\D_i)_{1\leq i \leq 5}$, such that $h_{\D_{i}} /h_{\D_{i+1}} = 2$.
Since the time discretisation is of order one, in $L^\infty_t(L^2_x)$-norm, we expect the error to decrease as
  \[
  \texttt{Error}\leq C_T\Delta t+C_S(k)h_{\D_{i}}^{k+2},
  \]
  where $C_T,C_S(k)>0$ are multiplicative constants respectively related to time and space discretisations, with $k\mapsto C_S(k)$ decreasing. We have $h_{\D_i}=h_{\D_1}/2^{i-1}$, so to balance the time and space contributions of the error upper bound, we need to take
  \[
  \Delta t(i,k)\sim\frac{\Delta t(1,k)}{2^{(i-1)(k+2)}},
    \]
    where $\Delta t(1,k)=\frac{C_S(k)}{C_T}h_{\D_1}^{k+2}$. For the values of $k$ we consider, we assume (and we verify in practice that it is relevant) that $\Delta t(1,k)\geq 0.05/2^{k+2}$ (see~\cite{ADPRu:17} for a theoretical study of $k\mapsto C_S(k)$ in the HHO context). Thus, for given $i$ and $k$, we define our (maximal) time step as
    \[
	\Delta t(i,k) = \frac{0.05}{ 2^{i(k+2)}}.
\]

On Figure~\ref{C4:fig:CV:evol:order}, we plot the relative errors as functions of the mesh size $h_\D$ for $k\in\{0,1,2,3\}$.
For completeness, we also include the scheme \texttt{nlhho\_2\_0} (i.e.~with $\varepsilon = 0$ in~\eqref{C4:sch:stationnary} for $k=2$) in our comparison.
First, we observe that \texttt{nlhho\_2} and \texttt{nlhho\_2\_0} have the same behaviour (the two plots are superimposed). Tests with other values of $k$, not shown here, indicate that the influence of $\varepsilon$ ($0$ or $1$) on the accuracy of the scheme is not noticeable. 
Second, we see that, as one could expect, the method \texttt{nlhho\_k} converges at order $k+2$ in $L^2_t(L^2_x)$-norm. In the $L^2_t(H^1_x)$-norm, if the expected convergence order of $k+1$ is attained for $k=0$ and $k=1$, then some sort of saturation appears for $k=2$ and $k=3$. Since this saturation does not show up in $L^2_t(L^2_x)$-norm, we suspect this might be due to our definition~\eqref{eq:gradd} of the discrete density gradient. Indeed, remark that, at the discrete level, the chain rule is violated, thus~\eqref{eq:gradd} is not exactly a discrete version of $\nabla\e^\ell$.

We now study efficiency, that is to say accuracy for a given computational cost.
On Figure~\ref{C4:fig:CV:evol:cost}, we plot the relative errors as functions of the simulation walltime (in $s$).
Here again, the results for the schemes \texttt{nlhho\_2} and \texttt{nlhho\_2\_0} are superimposed.
It is quite remarkable to observe that, even with a low-order discretisation in time, significant efficiency gains can be obtained using a \rev{larger} value of $k$\rev{, at least for values of $k\leq 2$}.
The gain is expected to be even larger after parallelising the local computations.
Of course, the use of higher-order time-stepping methods should also lead to significant gains of efficiency. This will be investigated in future works.

\begin{figure}[h]
\begin{minipage}[c]{.51\linewidth}
\begin{tikzpicture}[scale= 0.86]
        \begin{loglogaxis}[
            legend style = { 
              at={(0.5,1.1)},
              anchor = south,
              tick label style={font=\footnotesize},
              legend columns=4
            },ylabel=\small{Relative $L^2_t(L^2_x)$-error},xlabel=\small{Walltime (in $s$)}
          ]          
          \addplot[color=blue,mark=triangle*] table[x=time_cost,y=error_sol] {data/time_CV_adpt/errors_hho_0};          
          \addplot[color=red,mark=square*] table[x=time_cost,y=error_sol] {data/time_CV_adpt/errors_hho_1};
          \addplot[color=OliveGreen,mark=*] table[x=time_cost,y=error_sol] {data/time_CV_adpt/errors_hho_2};
          \addplot[color=black,mark=star] table[x=time_cost,y=error_sol] {data/time_CV_adpt/eps=0/errors_hho_2};
          \addplot[color=brown, mark=diamond*] table[x=time_cost,y=error_sol] {data/time_CV_adpt/errors_hho_3};  
          \legend{ \tiny \texttt{nlhho\_0} , \tiny \texttt{nlhho\_1} , \tiny \texttt{nlhho\_2}, \tiny \texttt{nlhho\_2\_0}, \tiny \texttt{nlhho\_3}  }      
        \end{loglogaxis}
      \end{tikzpicture}    
\end{minipage}
\begin{minipage}[c]{.51\linewidth}
\begin{tikzpicture}[scale= 0.86]
        \begin{loglogaxis}[
            legend style = { 
              at={(0.5,1.1)},
              anchor = south,
              tick label style={font=\footnotesize},
              legend columns=4
            },ylabel=\small{Relative $L^2_t(H^1_x)$-error},xlabel=\small{Walltime (in $s$)}
          ]               
          \addplot[color=blue,mark=triangle*] table[x=time_cost,y=error_grad] {data/time_CV_adpt/errors_hho_0};          
          \addplot[color=red,mark=square*] table[x=time_cost,y=error_grad] {data/time_CV_adpt/errors_hho_1};          
          \addplot[color=OliveGreen,mark=*] table[x=time_cost,y=error_grad] {data/time_CV_adpt/errors_hho_2};
          \addplot[color=black,mark=star] table[x=time_cost,y=error_grad] {data/time_CV_adpt/eps=0/errors_hho_2};
          \addplot[color=brown, mark=diamond*] table[x=time_cost,y=error_grad] {data/time_CV_adpt/errors_hho_3};            
          \legend{\tiny \texttt{nlhho\_0} , \tiny \texttt{nlhho\_1} , \tiny \texttt{nlhho\_2}  , \tiny \texttt{nlhho\_2\_0}, \tiny \texttt{nlhho\_3} } 
        \end{loglogaxis}
      \end{tikzpicture}
\end{minipage}
\caption{\textbf{Accuracy vs.~computational cost.} Relative errors on triangular meshes.}
\label{C4:fig:CV:evol:cost}
\end{figure}
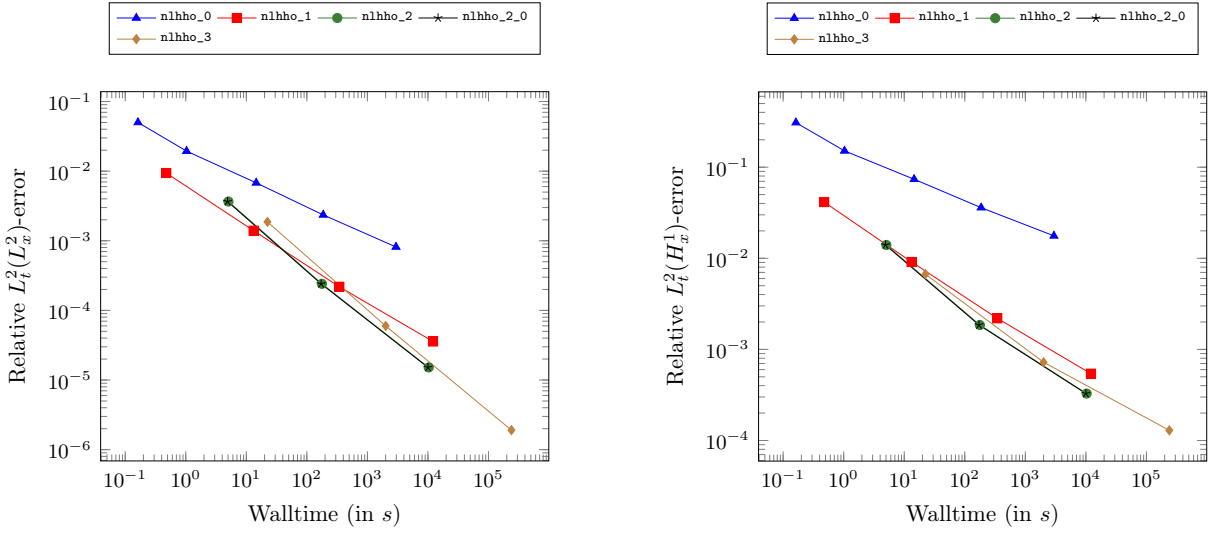

\begin{rem}[High-order schemes in time and space]
The extension of the nonlinear scheme~\eqref{C4:sch:AD} to arbitrary orders in time and space is a rather natural goal in order to achieve optimal efficiency.
However, even with a time discretisation of order $2$ (like for example \rev{Crank--Nicolson, which is perhaps the most natural extension to backward Euler}), there is currently no successful approach retaining the discrete entropy structure. 
Since this structure is the cornerstone of the analysis (including the existence of solutions), it is of utmost importance to preserve it.
Some numerical investigations on nonlinear entropic TPFA schemes for diffusive problems with BDF2 time discretisation have been performed in~\cite[Chapter 3]{Colin:16}, and indicate that such a time discretisation could \rev{also} be a good candidate, even in regard of long-time behaviour (see~\cite[Chapter 3.4.4]{Colin:16}).
An alternative approach is to consider space-time methods, as in~\cite{BPSto:22} in the context of conforming Galerkin discretisations of cross-diffusion systems. The extension of space-time techniques to polytopal grids is currently an active research area.
\end{rem}

\rev{For completeness, we finally perform simulations on distorted quadrangular meshes, and display the relative $L^2_t(L^2_x)$-errors on Figure~\ref{C4:fig:CV:quad}. Note that we use the same time step definition as for the previous simulations, whereas the initial mesh is coarser. As expected, the behaviour of the schemes is not strongly impacted by the mesh geometry, and \texttt{nlhho\_k} converges at order $k+2$ in $L^2_t(L^2_x)$-norm. When it comes to efficiency, increasing the value of $k$ leads to better accuracy for fixed computational cost, but the efficiency gain saturates for $k\geq 2$. It is also worth noting that on the coarsest mesh, \texttt{nlhho\_3} has to perform more time step reductions than the other schemes, because at some iterations the linear solver is unable to perform LU decomposition. These time step reductions occur not only at the beginning of the simulation, and are probably related to the bad conditioning of the system for high-order polynomials (we use here monomial bases). Based on these observations, using \texttt{nlhho\_2} seems to be a sound choice to optimise efficiency while ensuring a good numerical stability.
\begin{figure}[h]
\begin{minipage}[c]{.51\linewidth}
\begin{tikzpicture}[scale= 0.86]
        \begin{loglogaxis}[
            legend style = { 
              at={(0.5,1.1)},
              anchor = south,
              tick label style={font=\footnotesize},
              legend columns=4
            },ylabel=\small{Relative $L^2_t(L^2_x)$-error},xlabel=\small{Mesh size $h_\D$}
          ]
          \addplot[color=blue,mark=triangle*] table[x=Meshsize,y=error_sol] {data/time_CV_quad/errors_hho_0};          
          \addplot[color=red,mark=square*] table[x=Meshsize,y=error_sol] {data/time_CV_quad/errors_hho_1}; 
          \addplot[color=OliveGreen,mark=*] table[x=Meshsize,y=error_sol] {data/time_CV_quad/errors_hho_2};      
          \addplot[color=brown,mark=diamond*] table[x=Meshsize,y=error_sol] {data/time_CV_quad/errors_hho_3};
          \logLogSlopeTriangle{0.90}{0.2}{0.1}{2}{black};
          \logLogSlopeTriangle{0.90}{0.2}{0.1}{3}{black};
          \logLogSlopeTriangle{0.90}{0.2}{0.1}{4}{black};
          \logLogSlopeTriangle{0.90}{0.2}{0.1}{5}{black};
          \legend{\tiny \texttt{nlhho\_0} , \tiny \texttt{nlhho\_1} , \tiny \texttt{nlhho\_2}, \tiny \texttt{nlhho\_3} } 
        \end{loglogaxis}
      \end{tikzpicture}    
\end{minipage}
\begin{minipage}[c]{.51\linewidth}
\begin{tikzpicture}[scale= 0.86]
        \begin{loglogaxis}[
            legend style = { 
              at={(0.5,1.1)},
              anchor = south,
              tick label style={font=\footnotesize},
              legend columns=4
            },ylabel=\small{Relative $L^2_t(L^2_x)$-error},xlabel=\small{Walltime (in $s$)}
          ]          
          \addplot[color=blue,mark=triangle*] table[x=time_cost,y=error_sol] {data/time_CV_quad/errors_hho_0};          
          \addplot[color=red,mark=square*] table[x=time_cost,y=error_sol] {data/time_CV_quad/errors_hho_1};
          \addplot[color=OliveGreen,mark=*] table[x=time_cost,y=error_sol] {data/time_CV_quad/errors_hho_2};
          \addplot[color=brown, mark=diamond*] table[x=time_cost,y=error_sol] {data/time_CV_quad/errors_hho_3};  
          \legend{ \tiny \texttt{nlhho\_0} , \tiny \texttt{nlhho\_1} , \tiny \texttt{nlhho\_2} , \tiny \texttt{nlhho\_3}  }      
        \end{loglogaxis}
      \end{tikzpicture}    
\end{minipage}
\caption{\rev{\textbf{Accuracy on general meshes.} Relative $L^2_t(L^2_x)$-error on distorted quadrangular meshes.}}
\label{C4:fig:CV:quad}
\end{figure}
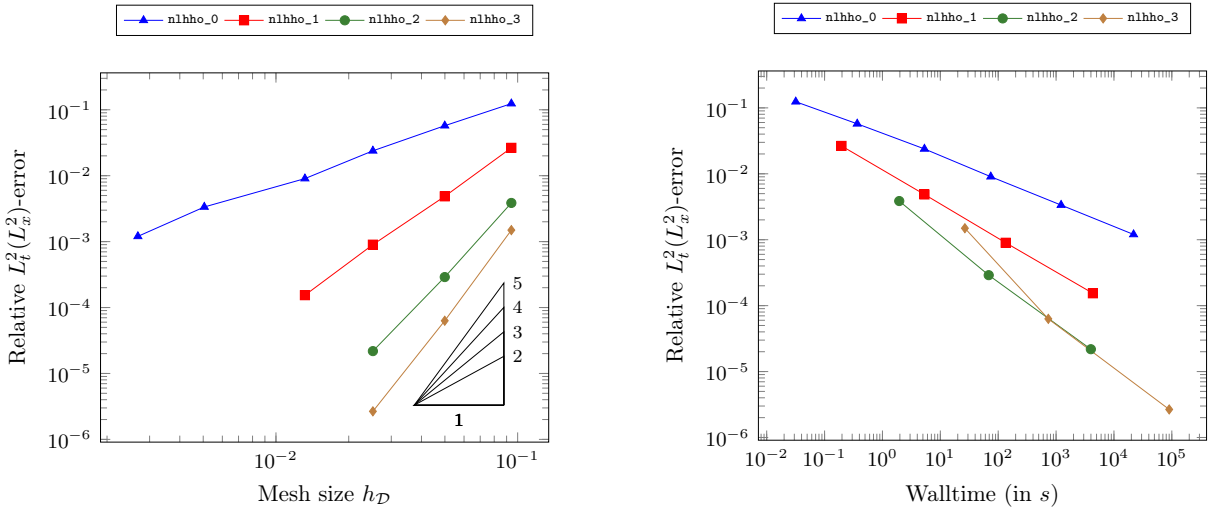}

\subsection{Discrete long-time behaviour}\label{C4:sec:LT}

We are now interested in the long-time behaviour of discrete solutions.

We first use the same test-case as in Section~\ref{C4:sec:CV}, but this time with an anisotropic tensor: we set $l_{x_1} =  10^{-2}$.
The corresponding steady-state is
\[
	u^\infty(x_1,x_2) = 2 C_1 \pi \e^{ x_1 - \frac{1}{2}  }.
\]
We compute the discrete solutions on the time interval $[0,350]$, with $\Delta t = 10^{-1}$, on two Kershaw meshes of sizes $0.02$ and $0.006$.
On Figure~\ref{C4:fig:longtime}, we display the evolution along time of the $L^1$-distance between the \rev{reconstructed} discrete densities and $u^\infty$, computed as 
\begin{equation} \label{eq:L1.dist}
	\|\mathfrak{u}^{\omega,n}_\M - u^\infty\|_{L^1(\Omega)}\qquad\text{and}\qquad\|\mathfrak{u}^n_\M-u^\infty\|_{L^1(\Omega)}
\end{equation}
for, respectively, the exponential fitting scheme, and the nonlinear scheme. We here focus on \texttt{expf\_1}, and on \texttt{nlhho\_k} for $k \in \lbrace 0,1,2 \rbrace$ (as well as on \texttt{nlhho\_1\_0}).
For all schemes, we observe the exponential convergence towards the thermal equilibrium, until machine precision is reached.
Remark that, for the test-case considered here, $\phi \in \poly^1(\Omega)$, therefore $\phi_\M = \phi$ for all $k\geq 0$.
It follows that $\mathfrak{u}_\M^\infty = u^\infty$ (recall that we always have $\mathfrak{u}^{\omega,\infty}_\M=u^\infty$). This is exactly what we observe in the numerical experiments.
As previously, \texttt{nlhho\_1} and \texttt{nlhho\_1\_0} exhibit an extremely similar behaviour.
Also, for $k \geq 1$, we observe that the decay rates are similar to the exact one $\alpha$, and do not seem to depend on the size of the mesh.
For $k=0$, the decay rate differs a bit from $\alpha$ on the coarsest mesh, but these two rates seem to coincide on a sufficiently refined mesh.

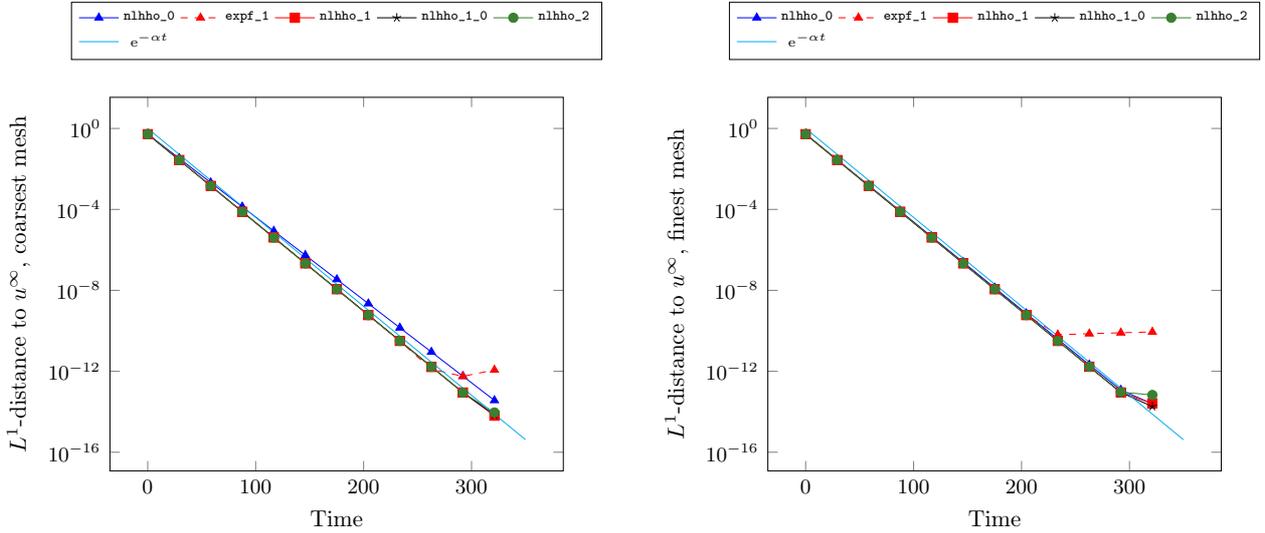
\begin{figure}[h]
\begin{minipage}[c]{.51\linewidth}
\begin{tikzpicture}[scale= 0.87]
        \begin{semilogyaxis}[
            legend style = { 
              at={(0.5,1.1)},
              anchor = south,
              tick label style={font=\footnotesize},
              legend columns=5
            },
            ylabel=\small{$L^1$-distance to $u^\infty$, coarsest mesh},
            xlabel=\small{Time}
            ]
          \addplot[color=blue,mark=triangle*] table[x=Time,y=Diff_L1] {data/tps_long/mesh4_1_1/time_nlhho_0};
          \addplot[color=red,dashed,mark=triangle*,mark options=solid] table[x=Temps,y=Diff_L1] {data/tps_long/mesh4_1_1/time_hho1};
          \addplot[color=red,mark=square*] table[x=Time,y=Diff_L1] {data/tps_long/mesh4_1_1/time_nlhho_1};
          \addplot[color=black,mark=star] table[x=Time,y=Diff_L1] {data/tps_long/mesh4_1_1/time_nlhho_1_0};             
          \addplot[color=OliveGreen,mark=*] table[x=Time,y=Diff_L1] {data/tps_long/mesh4_1_1/time_nlhho_2};       				
          \addplot[cyan] coordinates {
			(0,1)
			(350,4.1484969e-16)
	  };
          \legend{ \tiny \texttt{nlhho\_0},\tiny \texttt{expf\_1},\tiny \texttt{nlhho\_1}, \tiny \texttt{nlhho\_1\_0}, \tiny \texttt{nlhho\_2}, 
           \tiny $ \e^{- \alpha t} $}          
	      \end{semilogyaxis}
      \end{tikzpicture}    
\end{minipage}
\begin{minipage}[c]{.51\linewidth}
\begin{tikzpicture}[scale= 0.87]
        \begin{semilogyaxis}[
            legend style = { 
              at={(0.5,1.1)},
              anchor = south,
              tick label style={font=\footnotesize},
              legend columns=5
            },ylabel=\small{$L^1$-distance to $u^\infty$, finest mesh},xlabel=\small{Time}
          ]
          \addplot[color=blue,mark=triangle*] table[x=Time,y=Diff_L1] {data/tps_long/mesh4_1_4/time_nlhho_0};
          \addplot[color=red,dashed,mark=triangle*,mark options=solid] table[x=Temps,y=Diff_L1] {data/tps_long/mesh4_1_4/time_hho1};
          \addplot[color=red,mark=square*] table[x=Time,y=Diff_L1] {data/tps_long/mesh4_1_4/time_nlhho_1};
          \addplot[color=black,mark=star] table[x=Time,y=Diff_L1] {data/tps_long/mesh4_1_4/time_nlhho_1_0}; 
          \addplot[color=OliveGreen,mark=*] table[x=Time,y=Diff_L1] {data/tps_long/mesh4_1_4/time_nlhho_2};
          \addplot[cyan] coordinates {
			(0,1)
			(350,4.1484969e-16)
			};
          \legend{ \tiny \texttt{nlhho\_0},\tiny \texttt{expf\_1}, \tiny \texttt{nlhho\_1}, \tiny \texttt{nlhho\_1\_0}, \tiny \texttt{nlhho\_2}, 
           \tiny $ \e^{- \alpha t} $ }       
        \end{semilogyaxis}
      \end{tikzpicture}
\end{minipage}
\caption{\textbf{Long-time behaviour of discrete solutions.} $L^1$-distance to $u^\infty$ on Kershaw meshes.}
\label{C4:fig:longtime}
\end{figure}

As a last test-case, we consider an advective potential and an anisotropy tensor set to 
\[
	 \phi(x_1,x_2) =  -\frac{1}{2} \log \left ( 
	1 + (x_1-x_2)^2 + 3 x_2^2 \right )
\quad \text{ and } \quad
\Lambda = 
	\begin{pmatrix}
		10^3 & 0 \\ 
		0 & 1
	\end{pmatrix}.
\]
Our initial datum is 
\[
	u^{0}(x_1,x_2) = 1 + \frac{1}{2} \cos( 2 \pi x_1 ) \sin(2 \pi x_2).
\]
The corresponding thermal equilibrium therefore reads 
\[
 u^\infty(x_1,x_2) = \frac{1}{ \int_{(0,1)^2} \e^{-\phi}}\sqrt{1 + (x_1-x_2)^2 + 3 x_2^2 }.
\]
Remark that the potential $\phi$ is not (piecewise) polynomial.
As previously, we investigate the long-time behaviour of the schemes. We compute the discrete solutions on the time interval $[0,5]$, with $\Delta t = 0.2$, on two distorted quadrangular meshes featuring, respectively, 64 and 1024 cells.
On Figure~\ref{C4:fig:longtime:nonpoly}, we display the evolution of the $L^1$-distance to equilibrium, as defined in~\eqref{eq:L1.dist}, for both the \texttt{expf\_k} and \texttt{nlhho\_k} schemes, for $k\in\{0,1,2,3\}$. For all schemes, we observe the exponential convergence towards the thermal equilibrium, until some precision is reached. For the exponential fitting schemes, machine precision is attained (which is expected since $\mathfrak{u}^{\omega,\infty}_\M=u^\infty$), whereas for the nonlinear schemes (for which $\mathfrak{u}^\infty_\M$ is an approximation of $u^\infty$), the precision increases, as expected, with the polynomial degree and as the mesh is refined. Also, all schemes with $k\geq 1$ seem to exhibit a similar, meshsize-independent decay rate. For $k=0$, the decay rate seems slightly sensitive to the mesh size, but tends to reach the expected value on a sufficiently refined mesh.

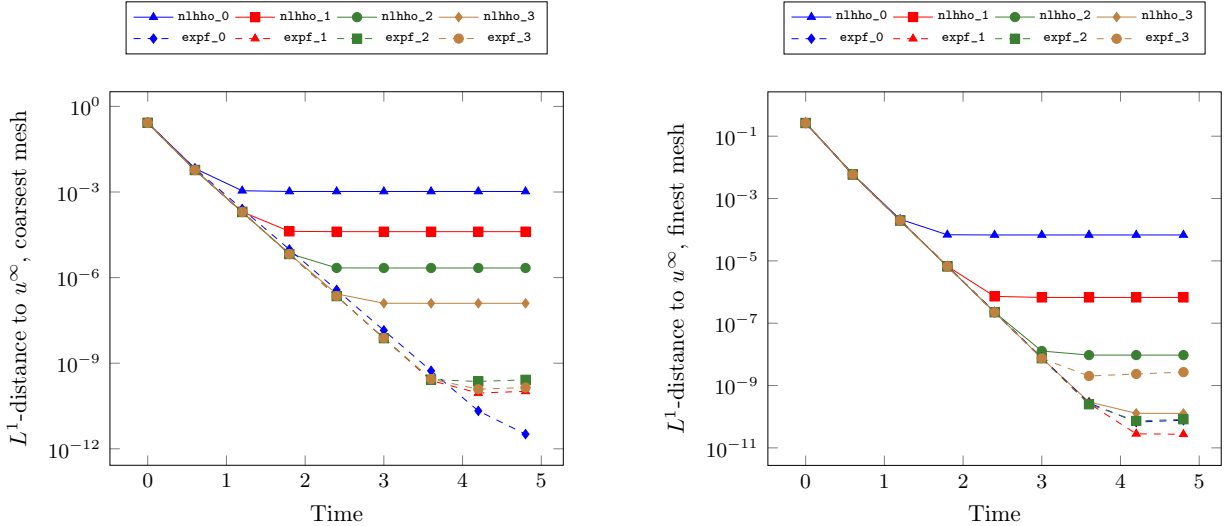
\begin{figure}[h]
\begin{minipage}[c]{.51\linewidth}
\begin{tikzpicture}[scale= 0.87]
        \begin{semilogyaxis}[
            legend style = { 
              at={(0.5,1.1)},
              anchor = south,
              tick label style={font=\footnotesize},
              legend columns=4
            },
            ylabel=\small{$L^1$-distance to $u^\infty$, coarsest mesh},
            xlabel=\small{Time}
            ]           
          \addplot[color=blue,mark=triangle*] table[x=Time,y=Diff_L1] {data/tps_long_nonpoly/mesh_quad_2/time_nlhho_0};                       
          \addplot[color=red,mark=square*] table[x=Time,y=Diff_L1] {data/tps_long_nonpoly/mesh_quad_2/time_nlhho_1};               
          \addplot[color=OliveGreen,mark=*] table[x=Time,y=Diff_L1] {data/tps_long_nonpoly/mesh_quad_2/time_nlhho_2};   
          \addplot[color=brown,mark=diamond*] table[x=Time,y=Diff_L1] {data/tps_long_nonpoly/mesh_quad_2/time_nlhho_3}; 	
          \addplot[color=blue,dashed,mark=diamond*,mark options=solid] table[x=Temps,y=Diff_L1] {data/tps_long_nonpoly/mesh_quad_2/time_hho0};                       
          \addplot[color=red,dashed,mark=triangle*,mark options=solid] table[x=Temps,y=Diff_L1] {data/tps_long_nonpoly/mesh_quad_2/time_hho1};               
          \addplot[color=OliveGreen,dashed,mark=square*,mark options=solid] table[x=Temps,y=Diff_L1] {data/tps_long_nonpoly/mesh_quad_2/time_hho2};   
          \addplot[color=brown,dashed,mark=*,mark options=solid] table[x=Temps,y=Diff_L1] {data/tps_long_nonpoly/mesh_quad_2/time_hho3}; 				
          \legend{ \tiny \texttt{nlhho\_0}, \tiny \texttt{nlhho\_1}, \tiny \texttt{nlhho\_2}, \tiny \texttt{nlhho\_3}, \tiny \texttt{expf\_0}, \tiny \texttt{expf\_1}, \tiny \texttt{expf\_2}, \tiny \texttt{expf\_3}}          
	      \end{semilogyaxis}
      \end{tikzpicture}    
\end{minipage}
\begin{minipage}[c]{.51\linewidth}
\begin{tikzpicture}[scale= 0.87]
        \begin{semilogyaxis}[
            legend style = { 
              at={(0.5,1.1)},
              anchor = south,
              tick label style={font=\footnotesize},
              legend columns=4
            },ylabel=\small{$L^1$-distance to $u^\infty$, finest mesh},xlabel=\small{Time}
          ]          
          \addplot[color=blue,mark=triangle*] table[x=Time,y=Diff_L1] {data/tps_long_nonpoly/mesh_quad_4/time_nlhho_0};
          \addplot[color=red,mark=square*] table[x=Time,y=Diff_L1] {data/tps_long_nonpoly/mesh_quad_4/time_nlhho_1};               
          \addplot[color=OliveGreen,mark=*] table[x=Time,y=Diff_L1] {data/tps_long_nonpoly/mesh_quad_4/time_nlhho_2};   
          \addplot[color=brown,mark=diamond*] table[x=Time,y=Diff_L1] {data/tps_long_nonpoly/mesh_quad_4/time_nlhho_3};           
          \addplot[color=blue,dashed,mark=diamond*,mark options=solid] table[x=Temps,y=Diff_L1] {data/tps_long_nonpoly/mesh_quad_4/time_hho0};                       
          \addplot[color=red,dashed,mark=triangle*,mark options=solid] table[x=Temps,y=Diff_L1] {data/tps_long_nonpoly/mesh_quad_4/time_hho1};               
          \addplot[color=OliveGreen,dashed,mark=square*,mark options=solid] table[x=Temps,y=Diff_L1] {data/tps_long_nonpoly/mesh_quad_4/time_hho2};   
          \addplot[color=brown,dashed,mark=*,mark options=solid] table[x=Temps,y=Diff_L1] {data/tps_long_nonpoly/mesh_quad_4/time_hho3}; 				
          \legend{ \tiny \texttt{nlhho\_0}, \tiny \texttt{nlhho\_1}, \tiny \texttt{nlhho\_2}, \tiny \texttt{nlhho\_3}, \tiny \texttt{expf\_0}, \tiny \texttt{expf\_1}, \tiny \texttt{expf\_2}, \tiny \texttt{expf\_3}}         
        \end{semilogyaxis}
      \end{tikzpicture}
\end{minipage}
\caption{\textbf{Long-time behaviour of discrete solutions.} $L^1$-distance to $u^\infty$ on distorted quadrangular meshes.}
\label{C4:fig:longtime:nonpoly}
\end{figure}

\section{Conclusion}

In this paper, we have studied two arbitrary-order hybrid methods for the approximation of linear, anisotropic\rev{, potential-driven} advection-diffusion equations on general polytopal meshes. 
The first one is a linear scheme, which is based on the exponential fitting strategy, whereas the second is a nonlinear scheme, whose building principles are adapted from the low-order constructions of~\cite{CaGui:17,CCHKr:18,CHHLM:22}.
We proved that both schemes admit solutions, possess a discrete entropy structure, and preserve the mass, the thermal equilibrium, and the long-time asymptotics.
Moreover, the solutions to the nonlinear scheme are positive by construction.
We have validated these theoretical results on a set of numerical test-cases.
We have unraveled the positivity violation of the linear methods, which justifies the use of (more costly) nonlinear methods. 
In the meantime, the use of nonlinear schemes with polynomial unknowns of higher degree results in an important gain of efficiency (accuracy vs.~computational cost).
These results confirm the benefits of using high-order nonlinear schemes in order to get reliable approximations of dissipative problems.
Future research directions include a full analysis of the nonlinear scheme, in particular of its convergence (with respect to the discretisation parameters) and time-asymptotic properties, as well as the development of similar schemes for more complex, nonlinear problems\rev{, like semiconductor models (based on~\cite{Moatt:23SC})}.

\clearpage

\ifJournal

\begin{acknowledgements}
  The authors would like to thank the anonymous reviewers for their remarks and suggestions which helped improving the quality of the presentation.
The authors also thank J.~Droniou for his insightful comments about this work, and for pointing out a simplification of the proof of Lemma~\ref{C4:lem:BbE}. The authors finally thank C.~Chainais-Hillairet and M.~Herda for fruitful discussions on the topic.
This research was funded in part by the Austrian Science Fund (FWF) project~\href{https://doi.org/10.55776/F65}{10.55776/F65}.
The authors also acknowledge support from the LabEx CEMPI (ANR-11-LABX-0007).  
\end{acknowledgements}

\bibliographystyle{ws-m3as}

\else

\section*{Acknowledgements}
\rev{The authors would like to thank the anonymous reviewers for their remarks and suggestions which helped improving the quality of the presentation.}
The authors also thank \rev{J.~Droniou} for his insightful comments about this work, and for pointing out a simplification of the proof of Lemma~\ref{C4:lem:BbE}. The authors finally thank C.~Chainais-Hillairet and M.~Herda for fruitful discussions on the topic.
This research was funded in part by the Austrian Science Fund (FWF) project~\href{https://doi.org/10.55776/F65}{10.55776/F65}.
The authors also acknowledge support from the LabEx CEMPI (ANR-11-LABX-0007).

\bibliographystyle{siam} 

\fi

\bibliography{sphho}

\end{document}